\documentclass[11pt,a4paper]{article}
\usepackage[T1]{fontenc}
\usepackage[utf8]{inputenc}
\usepackage{authblk}
\usepackage{diagbox}
\usepackage{amsmath, appendix, ulem, amsthm}
\usepackage{amssymb, xcolor}
\usepackage[margin=1in]{geometry}
\usepackage{mathrsfs}
\usepackage{graphicx}
\usepackage{subfig}
\usepackage{float}
\usepackage{epsf}
\usepackage{titletoc}
\usepackage{cases}
\usepackage[linesnumbered,ruled]{algorithm2e}
\usepackage{colortbl}
\usepackage[numbers,sort]{natbib}

\numberwithin{equation}{section}

\newtheorem{lem}{Lemma}[section]
\newtheorem{thm}{Theorem}[section]
\newtheorem{prop}{Proposition}[section]
\newtheorem{cor}{Corollary}[section]
\theoremstyle{remark}
\newtheorem{rmk}{Remark}[section]
\newcommand\keywords[1]{\textbf{Keywords}: #1}

\title{Transition behavior of the
waiting time distribution in a jumping model with the internal state}

\author[a]{Zhe Xue}
\author[b,d]{Yuan Zhang\thanks{Corresponding author: zhang\_probab@ruc.edu.cn}}
\author[c]{Zhennan Zhou}
\author[a]{Min Tang\thanks{Corresponding author: tangmin@sjtu.edu.cn}}

\affil[a]{School of Mathematical Sciences and Institute of Natural Sciences, MOE-LSC, CMA-Shanghai, Shanghai Jiao Tong University, China}
\affil[b]{Center for Applied Statistics and School of Statistics, Renmin University of China, China}
\affil[c]{Beijing International Center for Mathematical Research, Peking University, China}
\affil[d]{Pazhou Laboratory, Guangzhou 510330, China}

\date{}

\begin{document}

\maketitle

\begin{abstract}     
It has been noticed that when the waiting time distribution exhibits a transition from an intermediate time power law decay to a long-time exponential decay in the continuous time random walk model, a transition from anomalous diffusion to normal diffusion can be observed at the population level. However, the mechanism behind the transition of waiting time distribution is rarely studied. In this paper, we provide one possible mechanism to explain the origin of such transition. A jump model terminated by a state-dependent Poisson clock is studied by a formal asymptotic analysis for the time evolutionary equation of its probability density function. The waiting time behavior under a more relaxed setting can be rigorously characterized by probability tools. Both approaches show the transition phenomenon of the waiting time $T$, which is further verified numerically by particle simulations. Our results indicate that a small drift and strong noise in the state equation and a stiff response in the Poisson rate are crucial to the transitional phenomena. 
\end{abstract}


%----------------------------------------------------
\section{Introduction}
Diffusion processes are continuous-time, continuous-state processes whose sample paths are everywhere continuous but nowhere differentiable \cite{ibe2013markov,ikeda2014stochastic,pavliotis2014stochastic}. 
% In the real world, diffusion is everywhere. \textcolor{red}{For example, the Brownian motion, macro-molecule transport in biological cells, water seepage, and contaminants migration in porous media.} 
One can classify the diffusion process into normal diffusion and anomalous diffusion depending on whether the Fick's laws are obeyed \cite{https://doi.org/10.1002/andp.18551700105}. %\textcolor{red}{The normal diffusion processes is usually characterized by Fick's law which is not suitable for anomalous diffusion\cite{PCCPreview}.}
 Interestingly, a transition from anomalous diffusion to normal diffusion can be observed in many systems, for example, viscoelastic systems such as lipid bilayer membranes, systems of actively moving biological cells \cite{jeon2013noisy, molina2018crossover} or particles adsorbed in the internal walls of porous deposits \cite{reis2014crossover}, etc.
% We define the mean squared displacement(MSD) of a particle \cite{PCCPreview}
% \begin{equation}
%   \langle x^2(t) \rangle = \int_{-\infty}^{\infty} x^2 P(x,t) dt
% \end{equation}
% where $x(t)$ is the particle's position at time $t$ and $P(x,t)macro-molecule
% probability density function (PDF) for the particle. When the MSD is 
% linear in 
% time, namely $\langle x^2(t) \rangle \simeq Kt$, the process is called a 
% normal diffusion and $K$ is the diffusion coefficient. Besides, anomalous 
% diffusion refers to the power-law form 
% $\langle x^2(t) \rangle \simeq K_\alpha t^\alpha$ with the generalised 
% diffusion coefficient $K_\alpha$ and the anomalous diffusion exponent 
% $\alpha$ \cite{PCCPreview}. If $0<\alpha<1$, the process is a 
% sub-diffusion. If $\alpha > 1$, the particle undergoes a super-diffusion. Interestingly, we can see the transition from anomalous to normal 
% diffusion. 
% \textcolor{red}{No need to introduce MSD, we don't use it in our later description. More general discussion about what is diffusion, what is anomalous diffusion, the transition phenomena}

The most popular model to study the diffusion processes is the continuous time random walk (CTRW) model which was originally introduced by Montroll and Weiss \cite{montroll1965random, PCCPreview}. The CTRW model considers a particle that starts at the origin and consecutively jumps to different positions. The particle  
waits for a trapping time $\Delta t$ at each position and then jumps to another position whose distance from the previous position is $\Delta x$.  Here, $\Delta t$ and $\Delta x$ are two random variables (r.v.'s) whose probability density distributions (PDF) are respectively $\Psi(\Delta t)$ and  $\Phi(\Delta x)$, and there is no bias in the jumping direction \cite{haus1987diffusion, bouchaud1990anomalous, metzler2000random}. When the first moment of the waiting time and the variance of the jumping length are finite, the CTRW model gives normal diffusion. When the variance of the jumping length is finite but the waiting time distribution has a tail that decays according to the power law, i.e. 
\begin{equation}\label{Psi(Delta t)}
    \Psi(\Delta t) \sim \frac{1}{\Delta t^{1+\alpha}}, \quad \Delta t \to \infty,
\end{equation}
with $0<\alpha<1$, the CTRW model leads to sub-diffusion \cite{dentz2004time, PhysRevB.12.2455}. When all moments of $\Psi(\Delta t)$ are finite and $\Phi(\Delta x) \simeq 1/|\Delta x|^{1+\alpha}$ with $0<\alpha<2$, this will lead to a super-diffusion \cite{PhysRevE.49.4873, dentz2004time}. %The process was called a L\'{e}vy flight by Mandelbrot\cites{dentz2004time, mandelbrot1982fractal}. Besides, if we introduce a coupling between jump lengths and waiting times $\rho(\Delta x,\Delta t) = \Phi(\Delta x) \Psi(\Delta t) = \frac{1}{2} \Psi(\Delta t) \delta(\Delta x-v\Delta t)$ in which the velocity $v$ is given and $\Psi(\Delta t)$ is defined in \eqref{Psi(Delta t)} with $1<\alpha<2$. This will also lead to a super-diffusion\cite{PhysRevE.49.4873}.}
%\textcolor{blue}{add references. add sup-diffusion, $\Delta t$ finite, $\Delta x$ distributed in some way.}

The transition from anomalous diffusion to normal diffusion can be modeled by CTRW model as well. For example, the authors in \cite{dentz2004time}  study a CTRW model with the jumping length being the absolute value of a normal distribution and the waiting time distribution $\Psi (\Delta t)$ being
\begin{equation}\label{eq:ctrwtransition}
    \Psi(\Delta t) \propto \exp\left(-\frac{\Delta t}{t_2}\right) 
    \left(1+\frac{\Delta t}{t_1}\right)^{-(1+\beta)},
\end{equation}
where $t_1 \ll t_2$ are two time scales and $0< \beta <2$ is a constant. 
When $t_1 \ll \Delta t \ll t_2$, $\Psi(\Delta t)$ has 
a power-law decay with respect to $\Delta t$, i.e. $\Psi(\Delta t) \propto (\Delta t/t_1)^{-(1+\beta)}$, 
while when $t \gg t_2$, $\Psi(\Delta t)$ decreases exponentially fast. As has been pointed out in \cite{dentz2004time}, one can observe that $\Psi(\Delta t)$ as in \eqref{eq:ctrwtransition} induces 
the transition from anomalous diffusion to normal diffusion at the population level, which indicates that the transition from intermediate-time power-law decay to long-time exponential decay in waiting time distribution is highly related to the transitions from anomalous diffusion to normal diffusion.
% \textcolor{red}{In \cite{dentz2004time,jeon2013noisy}, when the observer is interested in the movement of the tracer within a living cell and the cell's connection to the microscope cover slip is disrupted, the data will show additional Brownian noise that comes from random cell movements superimposed on abnormal movements relative to the cell reference frame. This will lead to a turnover from the sub-diffusion to the Brownian motion time that goes by.}

One natural question is why the waiting time distribution may transit from the intermediate-time power-law decay to the long-time exponential decay. Motivated by the model and simulations in \cite{tu2005white} which studies the rotational directions of bacteria flagella, we propose a jump process controlled by the internal state and show that this transition can be induced by a small drift and a strong noise in the internal state. The construction of this model is also partially inspired by recent studies of the firing mechanism of neurons \cite{zhou2021investigating, liu2022rigorous}. We consider particles staying inside a potential well whose internal states $X_t$ evolve according to an Ornstein–Uhlenbeck (OU) process, where the strength of the drift is assumed to be of a lower order scale. The particles can jump outside of the potential well by a state-dependent Poisson process whose rate $\Lambda(x)$ equals to (or rapidly converges to) zero on one half of the $x$-axis and is uniformly bounded away from zero on the other half. We find out that the waiting time distribution of the particles staying inside the well exhibits a transition from an intermediate-time power-law decay to a long-time exponential decay.

We explore the transitional phenomenon from two approaches:
%show this transitional phenomenon by two methods,
one is a formal asymptotic analysis for the partial differential equation (PDE) describing the time evolution of the probability density function of $X_t$, and the other is a rigorous quantitative estimate by the probability tools. In the PDE approach, we can obtain the leading order behavior of the density distribution function in two different time regimes. The decay profile of the waiting time distribution in these two distinct regimes can be given analytically. However, this approach replies on explicit calculations that are only applicable to some special cases. The probability approach, on the other hand, applies to more general rate functions and similar results can be proved by estimating upper bounds and lower bounding for cumulative distribution of the stopping time. 

The paper is organized as follows. In Section 2, we first review the two state model for the \textit{E.Coli} flagella rotational direction in \cite{tu2005white} and propose a simplified one state jump model controlled by the internal state, then the main results are summarized. The PDE that describes the time evolution of the probability density function is studied in Section 3. We use Laplace transform and formal asymptotics to get the leading order of the waiting time distribution. In Section 4, the main theorem is proved by using the probability tool. To verify the theoretical results, Section 5 is devoted to the numerical simulations of the jump process. Finally, we summarize the paper and discuss future directions in Section 6.

\section{The model and the main results}\label{model}
\subsection{The model}

 %\textcolor{blue}{In \cite{tu2005white}, the authors consider the movement of E. Coli cells which each has} 
Each \textit{E.Coli} cell has 6-8 flagella that can rotate either clockwise (CW) or counter-clockwise (CCW). The rotational directions of flagella control the movement of the \textit{E.Coli} cells. When most of the motors rotate CCW, the flagella form a bundle and push the cell to run in a straight line. When one or more of the motors rotate CW, the cell tumbles without moving \cite{xue2018role}. 
 In \cite{tu2005white}, the authors model the switches between CW and CCW by a two state model, in which the switching rates are determined by the CheY-P concentration. %CheY-P is a protein that can bind to the flagella motor and increases CW rotation probability. In the numerical simulations in 
CheY-P is an intracellular protein whose concentration evolves according to an Ornstein–Uhlenbeck (OU) process such that 
\begin{equation}\label{[Y]}
    dY(t) = -\frac{Y(t)-Y_0}{\tau} dt + dB(t),
\end{equation}
where $Y(t)$ is the CheY-P concentration; $Y_0$ is a constant; $\tau$ is the CheY-P correlation time and $B(t)$ is the white noise. The switching rates from CCW to CW and CW to CCW are respectively 
\begin{equation}\label{k+}
    \Lambda_0(Y) \propto \exp \left(-\alpha_0 \frac{Y(t)-Y_0}{Y_0}\right), \quad 
    \Lambda_1(Y) \propto \exp \left(-\alpha_1 \frac{Y(t)-Y_0}{Y_0}\right),
\end{equation}
in which $\alpha_0$ and $\alpha_1$ are two constants. The authors find that when one uses large $\tau$ in \eqref{[Y]} and large $\alpha_{0}$ in \eqref{k+}, the distribution of the CCW duration time decays according to a power-law.

%We consider only the CCW state 
In this paper, we focus on the CCW state and study a simplified one state model which is a jump process inside one potential well. The jump process is controlled by the internal state $X_t$ that satisfies an OU process:
\begin{equation}\label{1}
    dX_t = -\epsilon X_tdt + \sqrt{2}dB_t.
\end{equation}
Here
$\epsilon$ is a small drift corresponding to large $\tau$ in \eqref{[Y]}, $B_t$ is the white noise and $Y_0$ is chosen to be $0$ in \eqref{[Y]}. $X_t$ is terminated by a state-dependent Poisson clock with a jumping
rate $\Lambda(X_t)$. We choose the jump rate $\Lambda(X_t)$ to increase rapidly from zero to a positive number in line with \eqref{k+}. This will allow us to imitate
% mimic 
the sharp transition when using large $\alpha_0$ in $\Lambda_0(Y)$. More specifically, let $\Lambda(X_t)$ be a nonnegative bounded measurable function supported on $[0,+\infty)$ such that 
\begin{equation}\label{Lambdax}
    \Lambda(X_t)  \geq C_{+} > 0, \quad
                        \text{for $X_t \geq 0$}, \qquad
                         \Lambda(X_t) = 0,\quad \text{for $X_t < 0$}.
\end{equation}

Consider a stopping time $T$:
\begin{equation}\label{eq:1passagetime}
    T =\inf \left\{t:\int_{0}^{t} \Lambda(X_s) ds > \Gamma \right\},
\end{equation}
where $\Gamma \sim \exp(1)$ is an exponentially distributed random variable with rate $\lambda_{\Gamma}=1$. $\Gamma$ is independent of $X_t$. Let $\mathcal{F}_t$ be the $\sigma$-algebra generated by $X_s$: $s \le t$, 
% From the definition of the exponential distribution,
one has 

\begin{equation*}
    \mathbb{P}(T > t | \mathcal{F}_t)  = \mathbb{P}\left( \left. \int_{0}^{t} \Lambda(X_s) ds \leq 
    \Gamma \right| \mathcal{F}_t\right) = e^{-\int_{0}^{t} \Lambda(X_s) ds}.
\end{equation*}
And given $\mathcal{F}_t$, the conditional jumping rate at time $t$ is given by
\begin{equation*}
     \frac{\frac{d \mathbb{P}(T \leq t | \mathcal{F}_t)}{dt}}{\mathbb{P}(T > t | \mathcal{F}_t)} = \frac{ \Lambda(X_t) e^{-\int_0^t \Lambda(X_s) ds}}{e^{-\int_0^t \Lambda(X_s) ds}}=\Lambda(X_t).
\end{equation*}

%One can easily get the distribution of the 
%waiting time before a jump occurrence is $\Lambda(X_t) \exp\left(-\int_0^t 
%\Lambda(X_s) ds\right)$.

At time $T$, the OU process is terminated and $X_t$ is set to a frozen state say $-\infty$ afterward. In other words, the above dynamics can 
be equivalently seen as the OU process is killed at a state-dependent 
Poisson rate $\Lambda(X_t)$. The termination of the OU process can be considered as 
the switching from CCW to CW in the two-state model\cite{tu2005white}. The killing time $T$ represents the CCW or CW duration time.

\paragraph{The PDE model.} 
Let $f(x,t)$ be the PDF of $X_t$. 
By Dynkin's formula \cite{ASPCJ}, $f(x,t)$ satisfies the following Fokker-Planck equation
\begin{equation}\label{eq:pde}
    \begin{cases}
        \frac{\partial f}{\partial t} - 
        \epsilon\frac{\partial}{\partial x}(xf) - 
        \frac{\partial^2f}{\partial x^2} = -\Lambda(x)f,	\\
        f(x,0) = \delta(0),
    \end{cases}
\end{equation} 
where $\Lambda(x)$ is the same as in \eqref{Lambdax} after replacing $X_t$ by $x$.
The probability that the particle has not jumped up to time $t$ is 
\begin{equation*}
    N(t) = \int_{-\infty}^{\infty} f(x,t) dx.
\end{equation*}
The PDF of the waiting time distribution is given by \cite{daly2007intertime, daly2006state, cox1962renewal}
\begin{equation*}
    n(t) = -\frac{d}{dt} N(t) = -\int_{-\infty}^{\infty} 
    \frac{\partial}{\partial t} f(x,t) dx.
\end{equation*}
From \eqref{eq:pde}, one has
\begin{equation}\label{eq:N(t)}
    n(t)= \int_{-\infty}^{\infty} \Lambda(x)f(x,t) dx,
\end{equation}
where $\Lambda(x)$ is defined as in \eqref{Lambdax}.

\subsection{Main results}

We study the waiting time distribution from two different perspectives: an asymptotic analysis based on the Fokker-Planck equation \eqref{eq:pde} and a rigorous proof by the probability tools. The main results are listed below.

In the PDE approach, we focus on the case that $\Lambda(x)$ is the following piece-wise constant function:
\begin{equation}\label{Lambdaxspecial}
     \Lambda(x)=\begin{cases}
                     1 & \text{$x \geq 0$}, \\
                     0 & \text{$x < 0$},
                \end{cases}
 \end{equation} and in this case the PDF of the waiting time distribution is further simplified to
 \[
n(t)=  \int_0^{\infty} f(x,t) dx.
 \]
 We have derived the leading order behavior of the waiting time PDF as follows.
\begin{prop}\label{prop21}
    Let $f(x,t)$ satisfy the Fokker-Planck equation \eqref{eq:pde} with the rate function given by \eqref{Lambdaxspecial}. When $\epsilon \to 0^{+}$, the waiting time distribution
    $n(t)$ as in \eqref{eq:N(t)} satisfies
    \begin{itemize}
          \item [(1)] 
          In an intermediate time scale $t = O(1)$ and $t < O(\frac{1}{\epsilon})$, the waiting time distribution $n(t) \sim t^{-\frac{3}{2}}$.
           
          \item [(2)]
          In a long time scale $t = O(\frac{1}{\epsilon^2})$, the waiting time distribution $n(t) 
          \sim e^{-t}$. 
        \end{itemize}
\end{prop}

For more general rate functions $\Lambda(X_t)$ as in \eqref{Lambdax},  we have the following estimates for the waiting time in similar time regimes.
\begin{thm}\label{mainthm1}
    Let $X_t$ satisfy the process \eqref{1} and be terminated by the state 
    dependent Poisson clock with intensity $\Lambda(X_t)$ that satisfies \eqref{Lambdax}. Then for the waiting time $T$ defined in \eqref{eq:1passagetime}, we have 
    \begin{itemize}
         
        \item [(a)]
        For any $\alpha \in (0,\frac{1}{6})$, there exits a positive constant
        $M=M(\alpha)<\infty$ depending only on $\alpha$ such that
        \begin{equation*}
            \mathbb{P}(T>t) \geq \frac{1}{2} t^{-\frac{1}{2}-\alpha}, \quad  \forall t \in 
            [M,\epsilon^{-\frac{1}{2}}],
        \end{equation*}
        for sufficiently small $\epsilon$.
        
        \item [(b)]
        There exists a positive constant $c>0$ such that 
        \begin{equation*}
            \mathbb{P}(T>t) \leq \exp(-ct),
        \end{equation*}
        with $t \in (\frac{1}{\epsilon^2},+\infty)$ for sufficiently 
        small $\epsilon$.
        
    \end{itemize} 
\end{thm}
It is worth noting that the proof of Theorem 2.1 relies on the following characteristics of rate function $\Lambda$: 
\begin{itemize}
    \item[1)] $\Lambda$ equals to or converges to $0$ on a half of the $x$-axis, which, combined with a sufficiently small drift term, provides an almost Brownian environment that generates the power law decay in the macroscopic scale;
    \item[2)] $\Lambda$ is uniformly bounded away from $0$ on the other half of the axis, which leads to a positive probability of triggering the stopping time for an excursion into this half. Thus the exponential decay in the long run follows by the gambler's ruin.
\end{itemize}

Our analytical results indicate that when $\tau$ is infinity in \eqref{[Y]}, i.e. when the CheY-P concentration has infinite time correlation, the waiting time distribution has a power-law decay tail, while when $\tau \gg 1$ but not infinity, there exhibits a transition from an intermediate-time power-law to a long-time exponential decay, the transitional time is at least at the order of $\sqrt{\tau}$. In fact, our numerical tests in Section \ref{sec:5} seem to suggest that the transition takes place around $O(\tau)$ time.
%\begin{rmk}
%\color{red}It is worth noting that the proof of Theorem 2.1 relies on the following characteristics of rate function $\Lambda$:
%     \begin{itemize}
%         \item [(1)] 
%         $\Lambda$ equals to/converges to $0$ in half of the $x$-axis, which, combined with a sufficiently small drift term, provides an almost Brownian environment that generates the power law decay in the macroscopic scale;

 %        \item [(2)]
 %        $\Lambda$ is uniformly bounded away from $0$ in the other half, which leads to a positive probability of triggering the stopping time for an excursion into this half. And thus the exponential decay in the long run follows by the gambler's ruin.
 %    \end{itemize}
%\end{rmk}

%\begin{rmk}
%    \color{red}Consider $X_t$ is a standard Brownian motion starting from $0$. $X_t$ is terminated by the state-dependent Poisson clock with the rate in \eqref{Lambdax}. For the waiting time $T$ defined in \eqref{eq:1passagetime}, we have for any $\alpha>0$, 
%    \begin{equation*}
%        \mathbb{P}(T>t) \geq \frac{1}{2} t^{-\frac{1}{2}-\alpha},
%    \end{equation*}
%    for all sufficiently large $t$.
%\end{rmk}

\section{Asymptotic analysis of the PDE model}\label{3}
%Proof of Proposition 2.1
To understand the behavior of $n(t)$ at different time scales, we look at the Laplace transform of $n(t)$. From the definition of the Laplace transform
\begin{equation*}
    \hat{n}(s) = \int_0^{\infty} n(t)e^{-ts} dt,
\end{equation*}
it is easy to find that $\hat{n}(\frac{s}{\alpha})=\alpha \int_0^{\infty}n(\alpha t)e^{-ts} dt$, with $\alpha$ being a constant. This indicates that when $\alpha \sim O(1)$, i.e. $\frac{s}{\alpha} \sim O(1)$, we are considering the intermediate time scale $t \sim O(1)$, while when $\alpha = \frac{1}{\epsilon^2}$, i.e. $\frac{s}{\alpha} \sim O(\epsilon^2)$, we are considering the long time scale $t \sim O(\frac{1}{\epsilon^2})$. Therefore in the subsequent part, we calculate the explicit expression of $\hat{n}(s)$ and find its asymptotic approximations when $s \sim O(1)$ and $s \sim O(\epsilon^2)$.

\subsection{The explicit formula of $\hat{n}(s)$}
Taking Laplace transform on both sides of \eqref{eq:pde} yields
\begin{equation}\label{eq:laplace1}
    \hat{f}^{''} + \epsilon x \hat{f}^{'} +[\epsilon - s - 
    \Lambda(x)] 
    \hat{f} = -f(x,0) = \delta(0),
\end{equation}
where $\hat{f}(x,s)$ is the Laplace transform of $f(x,t)$ and $'$ is the derivative with respect to $x$.

Letting $\hat{f}_{+}(x) = \hat{f}(x)|_{x \in (0,+\infty)}$ and  
$\hat{f}_{-}(x) = \hat{f}(x)|_{x \in (-\infty,0)}$, \eqref{eq:laplace1} can be rewritten into the following two equations:
\begin{subequations}\label{eq:hatf}
    \begin{equation}\label{eq:hatf+}
        \hat{f}_{+}^{''} + \epsilon x \hat{f}_{+}^{'} +[\epsilon - 
        (s+1)] 
        \hat{f}_{+} = 0 , 
    \end{equation}
    \begin{equation}\label{eq:hatf-}
        \hat{f}_{-}^{''} + \epsilon x \hat{f}_{-}^{'} +[\epsilon - s] 
        \hat{f}_{-} = 0 .
    \end{equation}
\end{subequations}
 $\hat{f}_{+}$ and $\hat{f}_{-}$ are connected at $x=0$ by
 \begin{equation}\label{connectionx=0}
    \lim_{x \to 0^+} \hat{f}_{+}(x) = \lim_{x \to 0^-} 
    \hat{f}_{-}(x), \quad
    \lim_{x \to 0^+} \hat{f}_{+}^{'}(x) - \lim_{x \to 0^-} 
    \hat{f}_{-}^{'}(x) = -1.
\end{equation}	
 We solve $\hat{f}_{+}$ and $\hat{f}_{-}$ in the subsequent part.
 
%  From the relationship between $n(t)$ and $f(x,t)$ as in \eqref{eq:N(t)}, one has
% \begin{equation}\label{Lphatns}
%     \hat{n}(s) = \int_0^{\infty} e^{-st} n(t) dt = 
%     \int_0^{\infty}\int_0^{\infty} e^{-st} f(x,t) dx dt =
%      \int_0^{\infty} \hat{f}(x,s) dx.
% \end{equation}
% Integrate \eqref{eq:hatf+} from $0$ to $+\infty$, one has
% \begin{equation*}
%     \int_0^{\infty} \hat{f}^{''}dx+ \epsilon \int_0^{\infty} (x\hat{f})^{'} 
%     dx - (s+1)\int_0^{\infty} \hat{f} dx  = 0.
% \end{equation*}
% Therefore, from the above expression and \eqref{Lphatns},
% \begin{equation}\label{eq:hatNpower}
%     \hat{n}(s) = \frac{1}{s+1} \left[\int_0^{\infty} \hat{f}^{''}dx+ 
%     \epsilon \int_0^{\infty} (x\hat{f})^{'} dx \right] \\
%     = \frac{1}{s+1}\left(\hat{f}_{+}^{'} 
%     \Big|_0^{+\infty} + \epsilon x\hat{f}_{+} \Big |_0^{+\infty}  
%     \right).
% \end{equation}
% As far as we can find $\hat{f}_{+}^{'}(+\infty)$, $\hat{f}_{+}^{'}(0)$ and $x\hat{f}_{+}(+\infty)$, $\hat{f}_{+}(0)$, $\hat{n}(s)$ can be given.

% First, we seek for the exact representation of $\hat{f}_{+}$. 
\paragraph{Solve $\hat{f}_{+}(x,s)$.}
By introducing $y=\sqrt{\epsilon}x$ and $\hat{H}_{+}(y,s) = 
e^{\frac{\epsilon}{4}x^2}\hat{f}_{+}(x,s)$, \eqref{eq:hatf+} can be written into
\begin{equation}\label{hatH+}
    \hat{H}_{+}^{''} - 
    \left(\frac{y^2}{4}+\frac{s+1}{\epsilon}-\frac{1}{2}\right)\hat{H}_{+} 
    = 0.
\end{equation}
\eqref{hatH+} is of the form of parabolic cylinder function in \eqref{PCFstandard}. From Appendix B, \eqref{hatH+} has two general 
solutions 
\begin{align}\label{hatH+1}
    \hat{H}_{+,1}(y) &={} U\left(\frac{s+1}{\epsilon}-\frac{1}{2},y\right) 
    = \sqrt{\pi}2^{-\frac{s+1}{2\epsilon}}
    \left[\frac{h_1(y)}{\Gamma(\frac{s+1}{2\epsilon}+\frac{1}{2})} - 
    \frac{\sqrt{2}h_2(y)}{\Gamma(\frac{s+1}{2\epsilon})}\right],
\end{align} 
\begin{align}\label{hatH+2}
    \hat{H}_{+,2}(y) &={} V\left(\frac{s+1}{\epsilon}-\frac{1}{2},y\right)  
    =  \frac{2^{-\frac{s+1}{2\epsilon}}}{\sqrt{\pi}
    \Gamma(1-\frac{s+1}{\epsilon})}
    \left[\sin(\frac{s+1}{2\epsilon}\pi)
    \Gamma(\frac{1}{2}-\frac{s+1}{2\epsilon})h_1(y)\right.  \notag\\
    &+{}\left.  \sqrt{2}\cos(\frac{s+1}{2\epsilon}\pi)
            \Gamma(1-\frac{s+1}{2\epsilon})h_2(y)\right].
\end{align}
where 
\begin{equation*}
    h_1(y) = e^{-\frac{\epsilon}{4}x^2} 
    M\left(\frac{s+1}{2\epsilon},\frac{1}{2},\frac{y^2}{2}\right), \quad 
    h_2(y) = ye^{-\frac{\epsilon}{4}x^2} 
            M\left(\frac{1}{2}+\frac{s+1}{2\epsilon},
            \frac{3}{2},\frac{y^2}{2}\right).
\end{equation*}
Here $M(a,b,x)$ is the Kummer's function, whose useful properties are listed in Appendix A; $U(a,y)$, $V(a,y)$ are called the parabolic cylinder functions whose properties are listed in Appendix B.

By the properties of the parabolic cylinder functions as in \eqref{U_infty} and \eqref{V_infty}, when $y \to \infty$, one gets
\begin{equation}\label{UandVprop}
    U\left(\frac{s+1}{\epsilon}-\frac{1}{2},y\right) \sim 
    e^{-\frac{y^2}{4}}y^{-\frac{s+1}{\epsilon}}, \quad
    V\left(\frac{s+1}{\epsilon}-\frac{1}{2},y\right) \sim 
    \sqrt{\frac{2}{\pi}} 
    e^{\frac{y^2}{4}}y^{\frac{s+1}{\epsilon}-1}.
\end{equation}
 Let the two general solutions to the equation for $\hat f_+$ in \eqref{eq:hatf+} be $\hat{f}_{+,1}(x)=e^{-\frac{\epsilon}{4}x^2}\hat{H}_{+,1}(y,s)$ and $\hat{f}_{+,2}(x)=e^{-\frac{\epsilon}{4}x^2}\hat{H}_{+,2}(y,s)$. Due to \eqref{UandVprop}, when $x \to +\infty$,
\begin{align*}
    \hat{f}_{+,1}(x) \sim e^{-\frac{\epsilon}{2}x^2} 
    (\sqrt{\epsilon}x)^{-\frac{s+1}{\epsilon}}, \quad
     \hat{f}_{+,2}(x) \sim \sqrt{\frac{2}{\pi}}
     (\sqrt{\epsilon}x)^{\frac{s+1}{\epsilon}-1}.
\end{align*}
It can be seen that $\hat{f}_{+,2}(x,s)$ goes to infinity when $x \to +\infty$. Thus $\lim_{x \to +\infty} \hat{f}_{+}(x) = 0$ yields
\begin{equation}\label{exacthatf+}
    \hat{f}_{+}(x) = 
    A(s)e^{-\frac{\epsilon}{4}x^2}\hat{H}_{+,1}(\sqrt{\epsilon}x) 
    = A(s) e^{-\frac{\epsilon}{4}x^2} 
    U\left(\frac{s+1}{\epsilon}-\frac{1}{2},\sqrt{\epsilon}x\right).
\end{equation}
Moreover, we can get the expression for $\hat{f}_{+}^{'}$ by using the property of $U(a,y)$ as in \eqref{derivate_U},
\begin{equation}
    U^{'}\left(\frac{s+1}{\epsilon}-\frac{1}{2},y\right) = 
    \frac{y}{2} U\left(\frac{s+1}{\epsilon}-\frac{1}{2},y\right)- 
    U\left(\frac{s+1}{\epsilon}-\frac{3}{2},y\right),
\end{equation}
one has
\begin{align}\label{exacthatf+'}
    \hat{f}_{+}^{'}(x) &={} -\frac{\epsilon}{2}x \hat{f}_{+}(x) + 
    A(s)\sqrt{\epsilon}e^{-\frac{\epsilon}{4}x^2}
    U^{'}\left(\frac{s+1}{\epsilon}-\frac{1}{2},\sqrt{\epsilon}x\right) 
    \notag\\
    &={} -A(s)\sqrt{\epsilon}e^{-\frac{\epsilon}{4}x^2}
    U\left(\frac{s+1}{\epsilon}-\frac{3}{2},\sqrt{\epsilon}x\right).
\end{align}

\paragraph{Solve $\hat{f}_{-}$.}
Similar as for $\hat{f}_{+}$, letting $y = -\sqrt{\epsilon}x$, one can find the general solution to \eqref{eq:hatf-} such that 
\begin{equation}\label{hatf-x}
    \hat{f}_{-}(x) = 
    B(s)e^{-\frac{\epsilon}{4}x^2}\hat{H}_{-,1}(-\sqrt{\epsilon}x) = B(s) e^{-\frac{\epsilon}{4}x^2} 
    U\left( \frac{s}{\epsilon}-\frac{1}{2}, -\sqrt{\epsilon} x \right),
\end{equation} 
where 
\begin{equation*}
    \hat{H}_{-,1}(y) = U\left( \frac{s}{\epsilon}-\frac{1}{2}, y \right)= \sqrt{\pi}2^{-\frac{s}{2\epsilon}}
    \left[\frac{h_3(y)}{\Gamma(\frac{s}{2\epsilon}+\frac{1}{2})} 
    - \frac{\sqrt{2}h_4(y)}{\Gamma(\frac{s}{2\epsilon})}\right],
\end{equation*}
with 
\begin{equation*}
        h_3(y) = e^{-\frac{\epsilon}{4}x^2} 
        M\left(\frac{s}{2\epsilon},\frac{1}{2},\frac{y^2}{2}\right), \quad
        h_4(y) = ye^{-\frac{\epsilon}{4}x^2} 
        M\left(\frac{1}{2}+\frac{s}{2\epsilon},
        \frac{3}{2},\frac{y^2}{2}\right).
    \end{equation*}
Using the property of $U$ in \eqref{derivate_U} in Appendix B, one gets the expression for $\hat{f}_{-}^{'}$ 
\begin{equation}\label{hatf'-x}
    \hat{f}_{-}^{'}(x) = B(s)\sqrt{\epsilon} e^{-\frac{\epsilon}{4}x^2} U\left( \frac{s}{\epsilon}-\frac{3}{2}, -\sqrt{\epsilon} x \right).
\end{equation}
\paragraph{Connecting $\hat{f}_{+}$ and $\hat{f}_{-}$.}
$A(s)$, $B(s)$ can be determined by the connection conditions in
\eqref{connectionx=0}.
Letting $x \to 0+$ in \eqref{exacthatf+} and \eqref{exacthatf+'}, the properties of $U$ in \eqref{U_0} and \eqref{derivate_U_0} give
\begin{align*}
    \hat{f}_{+}(0+) &={} \frac{A(s) 2^{-\frac{s+1}{2\epsilon}} \pi^{\frac{1}{2}}}{\Gamma\left(\frac{1}{2}+\frac{s+1}{2\epsilon}\right)},  \\
    \hat{f}_{+}^{'}(0+) &={} -\frac{A(s) 2^{\frac{1}{2}-\frac{s+1}{2\epsilon}}
    \pi^{\frac{1}{2}} \sqrt{\epsilon}}{\Gamma\left(\frac{s+1}{2\epsilon}\right)}.
\end{align*}
Similarly, when $x \to 0-$, by \eqref{U_0} and \eqref{derivate_U_0} again, \eqref{hatf-x} and \eqref{hatf'-x} give
\begin{align*}
    \hat{f}_{-}(0-) &={} \frac{B(s) 2^{-\frac{s}{2\epsilon}} \pi^{\frac{1}{2}}}{\Gamma\left(\frac{1}{2}+\frac{s}{2\epsilon}\right)}, 
    \\
    \hat{f}_{-}^{'}(0-) &={} \frac{B(s) 2^{\frac{1}{2}-\frac{s}{2\epsilon}}
    \pi^{\frac{1}{2}} \sqrt{\epsilon}}{\Gamma\left(\frac{s}{2\epsilon}\right).}
\end{align*}
Substituting the above four expressions into the connection conditions 
\eqref{connectionx=0}, we find 
\begin{equation*}
    A(s) = \frac{\sqrt{\frac{\pi}{\epsilon}}
    2^{\frac{s+1}{2\epsilon}-\frac{1}{2}}}
    {\frac{\pi}{\Gamma(\frac{s+1}{2\epsilon})}
    +\frac{\pi\Gamma(\frac{1}{2}
    +\frac{s}{2\epsilon})}{\Gamma(\frac{s}{2\epsilon})\Gamma(\frac{1}{2}+
    \frac{s+1}{2\epsilon})}}.	
\end{equation*}

In summary, we get the explicit solution $\hat{f}(x,s)$ to \eqref{eq:hatf}-\eqref{connectionx=0}.

\paragraph{The explicit formula for $\hat{n}(s)$.}
Integrating \eqref{eq:hatf+} from $0$ to $+\infty$, one can get that
\begin{equation*}
    \int_0^{+\infty} \hat{f}_{+}^{''} dx + \epsilon \int_0^{+\infty} (x\hat{f}_{+}^{'}) dx
    -(s+1)\int_0^{+\infty} \hat{f}_{+} dx = 0.
\end{equation*}
As a result, from \eqref{eq:N(t)}
\begin{align}\label{hatnexact}
    \hat{n}(s) &={} \int_0^{+\infty} \hat{f}_{+} dx = \frac{1}{s+1} \left[\int_0^{+\infty} \hat{f}_{+}^{''} dx + \epsilon \int_0^{+\infty} (x\hat{f}_{+}^{'}) dx\right] \notag \\
    &={} \frac{1}{s+1} \left(\hat{f}_{+}^{'}|_0^{+\infty} + \epsilon x \hat{f}_{+}|_0^{+\infty}\right).
\end{align}
Therefore to get $\hat{n}(s)$, one needs only $\hat{f}_{+}^{'}(+\infty)$, $\hat{f}_{+}^{'}(0)$, $\lim_{x\to0+}\epsilon x \hat{f}_{+}(x)$ and $\lim_{x\to+\infty}\epsilon x \hat{f}_{+}(x)$.
From \eqref{exacthatf+}, \eqref{exacthatf+'} and the limiting behavior of $U$ in \eqref{U_0} and \eqref{U_infty}, we have 
\begin{equation}\label{hatf+infty}
    \lim_{x\to+\infty} \epsilon x \hat{f}_{+}(x) = 0, \quad
    \lim_{x\to+\infty} \hat{f}_{+}^{'}(x) = 0. 
\end{equation}	
and
\begin{equation}\label{hatf+0f'0}
    \lim_{x \to 0+}\epsilon x \hat{f}_{+}(x) = 0, \quad
    \lim_{x \to 0+}\hat{f}_{+}^{'}(x) = -\frac{A(s) 2^{\frac{1}{2}-\frac{s+1}{2\epsilon}}
    \pi^{\frac{1}{2}} \sqrt{\epsilon}}{\Gamma\left(\frac{s+1}{2\epsilon}\right)}.
\end{equation}
Substituting \eqref{hatf+0f'0}, \eqref{hatf+infty} into \eqref{hatnexact}, we can get 
\begin{align}\label{eq:hatn}
    \hat{n}(s) &={} -\frac{1}{s+1}\hat{f}_{+}^{'}(0+) 
    = \frac{1}{s+1} \frac{A(s) 2^{\frac{1}{2}-\frac{s+1}{2\epsilon}}
    \pi^{\frac{1}{2}} \sqrt{\epsilon}}{\Gamma\left(\frac{s+1}{2\epsilon}\right)} \notag \\
    &={} \frac{1}{(s+1)[1+\digamma(s;\epsilon)]}.
\end{align}
where 
\begin{equation}\label{digammas}
    \digamma(s;\epsilon) := \frac{\Gamma\left(\frac{s+1}{2\epsilon}\right)
    \Gamma\left(\frac{1}{2}+\frac{s}{2\epsilon}\right)}
    {\Gamma\left(\frac{s}{2\epsilon}\right)
    \Gamma\left(\frac{1}{2}+\frac{s+1}{2\epsilon}\right)}.
\end{equation}

\subsection{The asymptotic approximation when $s\sim O(1)$}
Considering the intermediate regime $s\sim O(1)$, we need to find the leading order approximation of $\hat{n}(s)$ in \eqref{eq:hatn} when $\epsilon\to 0$.
As $z \to +\infty$, Stirling's formula gives
\begin{equation*}
    \Gamma(z) \approx 
    \sqrt{\frac{2\pi}{z}}
    \left(\frac{z}{e}\right)^z\left[1+O(\frac{1}{z})\right], \quad Re(z)>0.
\end{equation*}
When $\epsilon \ll 1$, $s=O(1)$, one has the following approximations to the Gamma function terms in 
\eqref{digammas} such that
\begin{subequations}\label{Gammaapproxpowerlaw}
    \begin{equation}
        \Gamma\left(\frac{s+1}{2\epsilon}\right) = e^{-\frac{s+1}{2\epsilon}}
    \sqrt{\frac{4\pi\epsilon}{s+1}} 
    \left(\frac{s+1}{2\epsilon}\right)^{\frac{s+1}{2\epsilon}} 
    \left[1+O(\frac{2\epsilon}{s+1})\right],
    \end{equation}
    \begin{equation}
        \Gamma\left(\frac{1}{2}+\frac{s}{2\epsilon}\right) = 
    e^{-\frac{s+\epsilon}{2\epsilon}}
    \sqrt{\frac{4\pi \epsilon}{s+\epsilon}}
    \left(\frac{s+\epsilon}{2\epsilon}\right)^{\frac{s+\epsilon}{2\epsilon}}
     \left[1+O(\frac{2\epsilon}{s+\epsilon})\right],
    \end{equation}
    \begin{equation}
        \Gamma\left(\frac{s}{2\epsilon}\right) = e^{-\frac{s}{2\epsilon}}
    \sqrt{\frac{4\pi 
    \epsilon}{s}}\left(\frac{s}{2\epsilon}\right)^{\frac{s}{2\epsilon}}
    \left[1+O(\frac{2\epsilon}{s})\right],
    \end{equation}
    \begin{equation}
        \Gamma\left(\frac{1}{2}+\frac{s+1}{2\epsilon}\right) = 
    e^{-\frac{s+1+\epsilon}{2\epsilon}}
    \sqrt{\frac{4\pi \epsilon}{s+1+\epsilon}} 
    \left(\frac{s+1+\epsilon}{2\epsilon}\right)^{\frac{s+1+\epsilon}
    {2\epsilon}}\left[1+O(\frac{2\epsilon}{s+1+\epsilon})\right].
    \end{equation}
\end{subequations}

Thus
\begin{subequations}
    \begin{equation}\label{Gamma1}
        \Gamma\left(\frac{s+1}{2\epsilon}\right)
    \Gamma\left(\frac{1}{2}+\frac{s}{2\epsilon}\right)
    =4\pi \epsilon (s+1)^{\frac{s+1}{2\epsilon}-\frac{1}{2}} 
    (s+\epsilon)^{\frac{s}{2\epsilon}}
    \left(\frac{e}{2\epsilon}\right)^{\frac{2s+1+\epsilon}{2\epsilon}}
    \left[1+O(\epsilon)\right],
    \end{equation} 
    and
    \begin{equation}\label{Gamma2}
        \Gamma\left(\frac{s}{2\epsilon}\right)
    \Gamma\left(\frac{1}{2}+\frac{s+1}{2\epsilon}\right) = 
    4\pi \epsilon s^{\frac{s}{2\epsilon}-\frac{1}{2}} 
    (s+1+\epsilon)^{\frac{s+1}{2\epsilon}} 
    \left(\frac{e}{2\epsilon}\right)^{\frac{2s+1+\epsilon}{2\epsilon}}
    \left[1+O(\epsilon)\right].
    \end{equation}
\end{subequations}

From \eqref{Gamma1} and \eqref{Gamma2}, one has
\begin{align}\label{digamma1}
    \digamma(s;\epsilon) &={}  
    \frac{4\pi \epsilon (s+\epsilon)^{\frac{s}{2\epsilon}} 
    (s+1)^{\frac{s+1}{2\epsilon}-\frac{1}{2}} 
    (\frac{e}{2\epsilon})^{\frac{2s+1+\epsilon}{2\epsilon}}
    [1+O(\epsilon)]}{4\pi \epsilon 
    s^{\frac{s}{2\epsilon}-\frac{1}{2}} 
    (s+1+\epsilon)^{\frac{s+1}{2\epsilon}} 
    (\frac{e}{2\epsilon})^{\frac{2s+1+\epsilon}{2\epsilon}}
    [1+O(\epsilon)]} \notag\\         
    &={} \frac{\sqrt{s}}{\sqrt{s+1}} 
    \frac{(1+\frac{\epsilon}{s})^{\frac{s}{2\epsilon}}}
    {(1+\frac{\epsilon}{s+1})^{\frac{s+1}{2\epsilon}}} 
    [1+O(\epsilon)]
    = \frac{\sqrt{s}}{\sqrt{s+1}} \frac{e^{\frac{s}{2\epsilon}\log(1+\frac{\epsilon}{s})}}
    {e^{\frac{s+1}{2\epsilon}\log(1+\frac{\epsilon}{s+1})}} [1+O(\epsilon)] \notag \\
    &={} \frac{\sqrt{s}}{\sqrt{s+1}} \frac{e^{\frac{1}{2}+\frac{\epsilon}{4s}+O(\epsilon^2)}}
    {e^{\frac{1}{2}+\frac{\epsilon}{4s+4}+O(\epsilon^2)}} [1+O(\epsilon)] 
    = \frac{\sqrt{s}}{\sqrt{s+1}} e^{\frac{\epsilon}{4(s+1)}+O(\epsilon^2)} [1+O(\epsilon)] \notag \\
    &={} \frac{\sqrt{s}}{\sqrt{s+1}} [1+\frac{\epsilon}{4(s+1)}+O(\epsilon^2)] [1+O(\epsilon)] = 
    \frac{\sqrt{s}}{\sqrt{s+1}} + O(\epsilon).
\end{align}

Therefore for $\hat{n}(s)$, we have the following approximation 
\begin{align*}
    \hat{n}(s) &={} 
    \frac{1}{(s+1)\left[1+\frac{\sqrt{s}}{\sqrt{s+1}}+O(\epsilon)\right]} 
    = \frac{1}{(s+1)\left[1+\frac{\sqrt{s}}{\sqrt{s+1}}\right]} 
    \frac{1}{1+O(\epsilon)} \\
    &={} \left(1-\sqrt{\frac{s}{s+1}}\right)(1+O(\epsilon)).
\end{align*}
 Define
\begin{equation*}
    \hat{n}_0(s) := 1-\sqrt{\frac{s}{s+1}}.
\end{equation*}
$\hat{n}_0(s)$ is the leading order term of $\hat{n}(s)$. Taking inverse Laplace 
transform yields
\begin{equation}\label{n0tpowerlaw}
    n_0(t) = \mathscr{L}^{-1}[\hat{n}_0(s)] = 
    e^{-\frac{t}{2}}[I_0(\frac{t}{2})-I_1(\frac{t}{2})] \sim 
    C_{p}t^{-\frac{3}{2}},
\end{equation}
where $I_0(t)$ and $I_1(t)$ are the first modified Bessel 
functions\cite{NISTDLMF} and $C_p$ is a constant.

When $s$ is at $O(\epsilon)$, the two terms $\Gamma(\frac{s}{2\epsilon})$ 
and $\Gamma(\frac{s}{2\epsilon}+\frac{1}{2})$ can no longer be approximated as in \eqref{Gammaapproxpowerlaw}. As a result, the power law decay in \eqref{n0tpowerlaw} is no longer valid when $t$ is larger than $O(\frac{1}{\epsilon})$.

\subsection{The asymptotic approximation when $s=O(\epsilon^2)$}
We then consider $s=O(\epsilon^2)$ and let $\tilde{s}=\epsilon^{-2}s$. Then $\tilde{s}=O(1)$ and $\hat{n}(s)$ becomes
\begin{equation}\label{expdecayhatn}
    \hat{n}(\tilde{s})= \frac{1}{\epsilon^2 \tilde{s}+1} 
    \frac{1}{1+\frac{\Gamma(\frac{\epsilon^2 \tilde{s}+1}{2\epsilon})
    \Gamma(\frac{1}{2}+\frac{\tilde{s}}{2}\epsilon)}
    {\Gamma(\frac{\tilde{s}}{2}\epsilon)
    \Gamma(\frac{1}{2}+\frac{\epsilon^2\tilde{s}+1}{2\epsilon})}} = 
    \frac{1}{(\epsilon^2 \tilde{s}+1)(1+\tilde{\digamma}(\tilde{s};\epsilon))}.
\end{equation}
where 
\begin{equation*}
    \tilde{\digamma}(\tilde{s};\epsilon):=\frac{\Gamma(\frac{\epsilon^2 \tilde{s}+1}{2\epsilon})
            \Gamma(\frac{1}{2}+\frac{\tilde{s}}{2}\epsilon)}
            {\Gamma(\frac{\tilde{s}}{2}\epsilon)
            \Gamma(\frac{1}{2}+\frac{\epsilon^2\tilde{s}+1}{2\epsilon})}.
\end{equation*}
When $\epsilon \to 0^+$, both $\frac{\epsilon^2 \tilde{s}+1}{2\epsilon}$ and 
$\frac{1}{2}+\frac{\epsilon^2\tilde{s}+1}{2\epsilon}$ will tend to 
$+\infty$. By Stirling's formula, one has
\begin{align*}
    \frac{\Gamma\left(\frac{\epsilon^2 \tilde{s}+1}{2\epsilon}\right)}
    {\Gamma\left(\frac{1}{2}+\frac{\epsilon^2\tilde{s}+1}{2\epsilon}\right)
    }
    &={} e^{\frac{1}{2}}(1+\frac{\epsilon}{\epsilon^2 
    \tilde{s}+1})^{-\frac{\epsilon^2 \tilde{s}+1}{2\epsilon}}
    \frac{\sqrt{\epsilon^2 \tilde{s}+1+\epsilon}}
    {\sqrt{\epsilon^2 \tilde{s}+1}}
     \left[1+O(\epsilon)\right]	\\
    &={} e^{\frac{\epsilon}{4(\epsilon^2 \tilde{s}+1)}+O(\epsilon^2)} 
    \frac{\sqrt{\epsilon^2 \tilde{s}+1+\epsilon}}{\sqrt{\epsilon^2 
    \tilde{s}+1}} \left[1+O(\epsilon)\right] \\
    &={} \frac{\sqrt{\epsilon^2 \tilde{s}+1+\epsilon}}{\sqrt{\epsilon^2 
    \tilde{s}+1}} \left[1+O(\epsilon)\right].
\end{align*}
On the other hand, by Euler's reflection formula, for a real and non integer $z$,
\begin{equation}\label{Eulerreflection}
    \Gamma(z)\Gamma(1-z) = \frac{\pi}{\sin{\pi z}}, \quad z \notin 
    \mathbb{Z}.
\end{equation}
which indicates that $\frac{1}{\Gamma(\frac{\tilde{s}}{2}\epsilon)}=\frac{\sin(\frac{\tilde{s}}{2}\epsilon \pi)
\Gamma(1-\frac{\tilde{s}}{2}\epsilon)}{\pi}$. Hence,
\begin{align*}
    \frac{\Gamma(\frac{1}{2}+\frac{\tilde{s}}{2}\epsilon)}
    {\Gamma(\frac{\tilde{s}}{2}\epsilon)} &={} 
    \pi^{-1} \Gamma(\frac{1}{2}+\frac{\tilde{s}}{2}\epsilon)
    \Gamma(1-\frac{\tilde{s}}{2}\epsilon) \sin(\frac{\tilde{s}}{2}\epsilon 
    \pi) \\
    &={} \pi^{-1} \Gamma(\frac{1}{2}) 
    \Gamma(1)[1+O(\epsilon)][\frac{\tilde{s}}{2}\epsilon \pi+O(\epsilon^2)] \\
    &={} \frac{\sqrt{\pi}}{2}\tilde{s}\epsilon + O(\epsilon^2),
\end{align*}
where we have used the Taylor expansions of $\Gamma(z)$ and $\sin(z)$.

Therefore, $\tilde{\digamma}(\tilde{s};\epsilon)$ can be approximated by
\begin{align*}
    \tilde{\digamma}(\tilde{s};\epsilon) &= \left\{\frac{\sqrt{\epsilon^2 
    \tilde{s}+1+\epsilon}}{\sqrt{\epsilon^2 \tilde{s}+1}} 
    \left[1+O(\epsilon)\right]\right\}\left[\frac{\sqrt{\pi}}{2}\tilde{s}\epsilon
     + 	
    O(\epsilon^2)\right] \\
    &= \epsilon \frac{\sqrt{\pi}}{2} \tilde{s} 
    \sqrt{1+\frac{\epsilon}{\epsilon^2 \tilde{s}+1}}+O(\epsilon^2).
\end{align*}
and $\hat{n}(\tilde{s})$ becomes
\begin{equation}
    \hat{n}(\tilde{s}) = \frac{1}{\epsilon^2\tilde{s}+1} \times
    \frac{1}{1+\epsilon \frac{\sqrt{\pi}}{2} \tilde{s} 
    \sqrt{1+\frac{\epsilon}{\epsilon^2 \tilde{s}+1}}+O(\epsilon^2)} 
    =\frac{1}{\epsilon^2\tilde{s}+1}\left(1+O(\epsilon)\right)
\end{equation}
Let
\begin{equation*}
    \hat{n}_0(\tilde{s}) := \frac{1}{\epsilon^2\tilde{s}+1},
\end{equation*}
\begin{equation}
    n_0(\tilde{t})= 
    \mathscr{L}^{-1}[\hat{n}_0(\tilde{s})] = 
    e^{-\frac{\tilde{t}}{\epsilon^2}}.
\end{equation}
 So that when $t \sim \frac{1}{\epsilon^2}$, $n_0(t)$ exhibits exponential decay.

Based on the formal calculations, we can infer that the waiting time PDF $n(t)$ behave power law decay in the middle and exponential decay when $t$ is large enough.

\section{The proof of Theorem 2.1}\label{4}
 In this section, our main 
results of the waiting time $T$ will be rigorously proved with the probability tool. First, we will rewrite Theorem 2.1 into two propositions which will be proved separately in the following subsections. 
%Then we prove the two propositions respectively.
\subsection{Exponential decay in the macroscopic time scale}
%in the end 

\begin{prop}\label{thm31}
    For $T$ defined in \eqref{eq:1passagetime}, there is a constant $c > 0$ such that
    \begin{equation}\label{eq:p1}
        P\left(T>t\right)\leq \exp(-ct),
    \end{equation}
    when $t \to \infty$ .
\end{prop}
In order to show \eqref{eq:p1}, it is equivalent to prove that $T$ is \textbf{exponentially integrable}, i.e. $\exists \ c>0$, such that $\mathbb{E}\left\lbrack \exp(c T)\right\rbrack < \infty$.
    Then \eqref{eq:p1} is an immediate result of Markov inequality as
    \begin{equation}
        \mathbb{P}\left(T>t \right)= \mathbb{P}\left(\exp(c T)>\exp(ct)\right) 
        \leq \dfrac{\mathbb{E}\left[\exp(cT)\right]}{\exp(ct)} \notag.
    \end{equation}
So in the rest of this subsection, we will concentrate on proving $T$ is exponentially integrable.

Now we present preliminary results that can be useful for our subsequent discussions.

\begin{lem}\label{lem31}
    For independent r.v.'s $X$ and $Y$, suppose they are both exponentially integrable. So does the sum $Z=X+Y$.
\end{lem}
\begin{proof}
    This simply follows from the fact that
    \begin{equation}
        \mathbb{E}\left[\exp(c Z)\right]= \mathbb{E}\left[\exp(c 
        X)exp(c Y)\right] = \mathbb{E}\left[\exp(c 
        X)\right]\mathbb{E}\left[\exp(c Y)\right],
    \end{equation}
    $\exists \ c > 0$ such that $\mathbb{E}\left\lbrack \exp(c Z)\right\rbrack < \infty$ . Since $X$ and $Y$ are independent of each other.
    \newline
\end{proof}

\begin{rmk}
    Lemma \ref{lem31} clearly holds true for all fixed finite summations.
\end{rmk}

\begin{lem}\label{lem32}
    Let $\{X_n\},n=1,2,\dots $, $X_n \geq 0$, be an independent and identically distributed(i.i.d.) sequence of exponentially integrable random variables, $N \sim G(p)$ is independent to 
    $\{X_n\}_{n=1}^{\infty}$, then we have
    \begin{equation*}
        Y=\sum_{n=1}^{N} X_n 
    \end{equation*}
    to be exponentially integrable. Here $G(p)$ is a geometric distribution with parameter $p$.
\end{lem}
\begin{proof}
    Note that $\exists \ c>0$, s.t. $\mathbb{E}\left[\exp(cX_1)\right] < \infty$. Then by dominated convergence theorem,
    \begin{equation*}
        \lim_{\lambda \rightarrow 0^+}\mathbb{E}\left[\exp(\lambda 
        X_1)\right]=1,
    \end{equation*}
    and thus there exists a $\lambda_0 > 0$ such that $\mathbb{E}\left[\exp(\lambda_0 X_1)\right] \in \left[1,(1+\frac{1}{1-p})/2 \right]$.
    
    Thus
    \begin{align*} 
        \mathbb{E}\left[\exp(\lambda_0 Y)\right]
        &={} \sum_{n=1}^{\infty} \left(\mathbb{E}\left[\exp(\lambda_0 
        X_1)\right]\right)^n \left(1-p\right)^{n-1} p, \\
        &\leq{} \sum_{n=1}^{\infty} \frac{(1+\frac{1}{1-p})^n}{2^n} 
        \left(1-p\right)^{n-1} p, \\
        &={}  \sum_{n=1}^{\infty} p(1-\frac{p}{2})^n < \infty. 
    \end{align*}
    So that $Y$ is exponentially integrable.
\end{proof}

We say a family of r.v.'s $\{X_i\}_{i \in I}$ are \textbf{uniformly exponentially integrable} if $\exists \ c >0,\ c_0 < \infty $,
\begin{equation*} 
    \mathbb{E}\left[\exp(cX_i)\right] \leq c_0, \quad \forall i \in I.
\end{equation*}
The following lemma shows that uniformly exponentially integrable leads to
\begin{equation*}
    \lim_{\lambda \to 0^{+}}\mathbb{E}\left[\exp(\lambda X_i) \right] = 1, 
\end{equation*}
uniformly for $i$. 

\begin{lem}\label{lem33}
    For r.v. $X \geq 0$, $c_0 = \mathbb{E}\left[c X\right] < \infty$ for some $c>0$, and any $\eta >0$, $\exists \ \delta = \delta(\eta,c,c_0)$,
    \begin{equation*} 
        \mathbb{E}\left[\exp(\delta X)\right] < 1+ \eta.
    \end{equation*}
\end{lem}
\begin{proof}
    Recalling $\mathbb{E}\left[(\exp(cX))\right] = c_0 < \infty $, then by Markov inequality one can get
    \begin{equation*}
        \mathbb{P}\left(X>t\right) \leq c_0\exp(-ct), \quad \forall t >0 .
    \end{equation*}
    Then for all integer $M > 0$, and $\delta < c/2$,
    \begin{align}
        \mathbb{E}\left[\exp(\delta X)\mathbf{1}_{\{X \geq M\}}\right] 
        &\leq{}  \sum_{k=M}^{\infty}\exp(\delta \cdot 
        (k+1))\mathbb{P}\left(k\leq X \leq k+1 \right) \notag ,\\
        &\leq{}  \sum_{k=M}^{\infty}\exp(\delta \cdot 
        (k+1))\mathbb{P}\left(X \geq k \right) \notag ,\\
        &\leq{} \sum_{k=M}^{\infty}\exp(\frac{c}{2}k+\frac{c}{2}))c_0\ 
        \exp(-ck) \notag ,\\
        &\leq{}  c_0\exp\left(\frac{c}{2}\right) 
        \sum_{k=M}^{\infty}\exp(-ck/2) \equiv I(M).
    \end{align}
    Thus $\exists \ M < \infty $ that depends on $c$ and $c_0$, such that $I(M) < \eta/2$. Then for this finite $M$, $\exists \ \delta < c/2$, s.t. 
    \begin{equation*}
        \exp(\delta M)< 1 + \eta/2.
    \end{equation*}
    Thus we can get 
    \begin{align*}
        1 \leq \mathbb{E}\left[exp(\delta X)\right]={} & 
        \mathbb{E}\left[exp(\delta X)\mathbb{I}_{\{X < 
        M\}}\right]+\mathbb{E}\left[exp(\delta X)\mathbb{I}_{\{X \geq 
        M\}}\right], \\
        <{} & 1 + \eta /2 +\eta /2 = 1+ \eta,
    \end{align*}
    where $\mathbb{I}$ represents the indicative function.
\end{proof}

Combining Lemma \ref{lem31} and \ref{lem33}, we can extend Lemma \ref{lem32} to uniformly exponentially integrable and independent, but not necessarily identically distributed r.v.'s
\begin{cor}\label{cor31}
    Let $\{X_n\}_{n=1}^{\infty}$ be an independent sequence of r.v.'s that 
    are uniformly exponentially integrable. $N \sim G(p)$ independent to 
    $\{X_n\}_{n=1}^{\infty}$. Then $Y = \sum_{n=1}^{N} X_n$ is also 
    exponentially integrable.
\end{cor}
\begin{proof}
	From Lemma \ref{lem33}, we have that
	\begin{equation}
		\lim_{\delta \to 0^{+}}\mathbb{E}\left[exp(\delta X_n) \right] = 1,	
	\end{equation}
uniformly for every $n$ as $\delta \to 0^{+}$.

As a result, $\exists \ \delta_0$ s.t. $\mathbb{E}\left[\exp(\delta_0 X_n)\right] \in [1,(1+\frac{1}{1+p})/2]$
 uniformly for every $n$. Thus
	\begin{align}
		\mathbb{E}\left[\exp(\delta_0 Y)\right]
		&={} \sum_{n=1}^{\infty} \left(\mathbb{E}\left[\exp(\delta_0 
		X_1)\right]\right)^n \left(1-p\right)^{n-1} p, \\
		&\leq{} \sum_{n=1}^{\infty} \frac{(1+\frac{1}{1-p})^n}{2^n} 	
		\left(1-p\right)^{n-1} p, \\
		&={}  \sum_{n=1}^{\infty} p(1-\frac{p}{2})^n < \infty.
	\end{align}
So we have proved the corollary.
\end{proof}

For all $y \in \mathbb{R}$, let $B_t^y$ be the standard Brownian Motion starting from $y$ and $X_t^y$ be the OU Process in \eqref{1} starting from $y$. Moreover, we define the following stopping times:
\begin{equation*}
	\tau_{es}^y = \inf\{t \geq 0,|X_t^y| \geq 2\},
\end{equation*}
which means the first time $X_t^y$ escapes interval $(-2,2)$
% escape $2$ or $-2$. 
\begin{equation*}
\tau_{r,+}^y = \inf\{t \geq 0,X_t^y = 1\},
\end{equation*}
\begin{equation*}
    \tau_{r,-}^y = \inf\{t \geq 0,X_t^y = -1\},
\end{equation*}
which means the first time $X_t^y$ return to $1$ and $-1$. For these stopping times, one has the following lemmas.	
\begin{lem}\label{lem34}
    $\tau_{r,+}^2,\tau_{r,-}^{-2}$ are both exponentially integrable.
\end{lem}
\begin{proof}
    By definition of the OU Process
    \begin{equation*} 
        X_t^2 = 2- \epsilon \int_{0}^{t} X_s^2 ds + \sqrt{2} B_t. 
    \end{equation*}
    Thus given the event $\tau_{r,+}^2 > t$, $X_s^2 > 1$ for all $s \leq t$, and
    \begin{equation*}
        1 < X_t^2 = 2 - \epsilon \int_{0}^{t}X_s^2 ds + \sqrt{2} B_t \leq 2 - 
        \epsilon t  + \sqrt{2} B_t, 
    \end{equation*}
    which implies that
    \begin{equation*}
        \mathbb{P}\left(\tau_{r,+}^2 >t \right)\leq \mathbb{P} 
        \left(\sqrt{2} B_t \geq \epsilon t-1\right) \leq  \mathbb{P}\left(B_t 
        \geq \frac{\sqrt{2}}{4} \epsilon t \right),
    \end{equation*}
    for  $t = O(\frac{1}{\epsilon})$. At the same time, one has
    \begin{align*}
        \mathbb{P}\left(B_t \geq \frac{\sqrt{2}}{4} \epsilon t \right) &={} 
        \int_{\frac{\sqrt{2}}{4} \epsilon t}^{+\infty} \frac{1}{\sqrt{2\pi t}} 
        e^{-\frac{x^2}{2t}} dx
        = \int_{\frac{\epsilon}{4}\sqrt{t}}^{+\infty} 
        \frac{1}{\sqrt{\pi}} e^{-y^2} dy 
        = \frac{1}{\sqrt{\pi}} 
        \int_{\frac{\epsilon}{4}\sqrt{t}}^{+\infty}  
        e^{\frac{-y^2}{2}} e^{\frac{-y^2}{2}} dy \\
        &\leq{} \frac{1}{\sqrt{\pi}} \exp(-\frac{\epsilon^2}{32}t) 
        \int_{\frac{\epsilon}{4}\sqrt{t}}^{+\infty}  
        e^{\frac{-y^2}{2}}  dy 
        \leq \tilde{C} \exp(-\frac{\epsilon^2}{32}t).
    \end{align*}
The same result holds for $\tau_{r,-}^{-2}$ as it is identically distributed as $\tau_{r,+}^2$ by symmetry. 
% So we have finished our proof.
\end{proof}

\begin{lem}\label{lem35}
    For the family of r.v.'s $\{\tau_{es}^y \}_{y \in [-2,2]}$, they are uniformly exponential integrable with respect to(w.r.t.) $y$.
\end{lem}
\begin{proof}
    Recalling the transition property of OU Process, for any $|y_0| \leq 2$
    \begin{equation*}
        X_1^{y_0} \sim N\left(e^{-\epsilon}y_0,\frac{1-e^{-2\epsilon}}{\epsilon}\right).
    \end{equation*}
    Noting that $\epsilon$ is sufficiently small, one gets $2 - y_0 e^{-\epsilon} < 4$. Thus we have
    \begin{equation}\label{pr4}
        \mathbb{P}\left(\tau_{es}^{y_0}<1\right) \geq{}  \mathbb{P}\left(|X_1^{y_0}|>2\right) 
        \geq \mathbb{P}\left(N(0,1)> 
        \frac{2-y_0e^{-\epsilon}}{\sqrt{\frac{1}{\epsilon}(1-e^{-2\epsilon})}}\right)=P_0 > 0.
    \end{equation}
    Note that the right hand side(RHS) of \eqref{pr4} does not depend on $y_0$. By 
    Markov Property\cite{J1962Markov} of OU Process, for any integer $n$, we have
    \begin{equation*} 
        \mathbb{P}\left(\tau_{es}^{y_0}>n\right) \leq \left[\sup_{|y_0| 
        \leq 2} \mathbb{P}\left(\tau_{es}^{y_0} \geq 1\right)\right]^n 
        \leq \left(1-P_0\right)^n.
    \end{equation*}
\end{proof}

On the other hand, for any OU Process starting from $[-1,1]$, it will with 
strictly positive probability stay within $[0,2]$ for at least a fixed
% a finite 
positive time before existing $[-2,2]$. To be specific,
\begin{lem}\label{lem36}
    $\exists \ P_0 > 0$, s.t. $\forall \ |y| \leq 1$
    \begin{equation}
        \mathbb{P}\left(\tau_{es}^y > 3,\quad 
        \int_{0}^{3}\mathbb{I}_{\{X_t^y \in [0,2]\}}dt \geq 1 \right) \geq 
        P_0. 
    \end{equation}
\end{lem}
\begin{proof}
    By symmetry, one may without loss of generality assume $y<0$, we 
    further define a stopping time $\tau_{0}^{y}=\inf\{t \geq 0, X_t^y =0\}$. Then by strong Markov property, as $X_t^y = 0$, we can consider a new OU 
    process starting from $0$. As a result, we have the following 
    inequality  
    \begin{equation}\label{lem36ineq}
        \mathbb{P}\left(\tau_{es}^y > 3,\int_{0}^{3}\mathbb{I}_{\{X_t^y 
        \in [0,2]\}}dt \geq 1 \right) \geq  A_1 A_2,
    \end{equation}
    where we set 
    \begin{equation*}
        A_1 = \mathbb{P}\left(\tau_{0}^y < \tau_{es}^y \wedge 
        1\right) =\mathbb{P}\left(\{\tau_{0}^y < \tau_{es}^y\} \cap \{\tau_0^y < 1 \} \right),
    \end{equation*}
    and
    \begin{equation*}
        A_2 = \min_{t \in [2,3]} \mathbb{P}\left( \tau_{es}^0 > t, 
        \int_{0}^{t}\mathbb{I}_{\{X_s^0 \in [0,2]\}}ds \geq t/2 \right).
    \end{equation*}
    Note that the event in RHS of \eqref{lem36ineq} is contained in the event of left hand side (LHS).
    %is a special case of the left hand side(LHS). 
    First, consider the stopping time $\tau_0^y$ less than 
    $\tau_{es}^y$ and $1$. Then as $X_t^y = 0$, we take the consideration of the new process starting from $0$. So we take account of a new stopping 
    time $\tau_{es}^0$. And it should be larger than $2$ because of 
    $\tau_{es}^y > 3$.
     
    For $A_1$, since $y<0$, and
    \begin{equation*}
        X_t^y = -\epsilon \int_{0}^{t}X_s^y ds + \sqrt{2} B_t^y,
    \end{equation*}
    one may define two stopping times
    \begin{equation*}
        \bar{\tau}_0^y= \inf \{t: \sqrt{2} B_t^y = 0\}, \quad \bar{\tau}_{-2}^y= 
        \inf \{t: \sqrt{2} B_t^y = -2\}.
    \end{equation*}
    Similar as before, in the event ${\tau_0^y>t}$, one has
    \begin{equation}\label{pr5}
        \sqrt{2} B_s^y=X_s^y + \epsilon \int_{0}^{s} X_h^y dh < 0,\quad \forall 
        s \leq t,
    \end{equation}
    which implies $\bar{\tau}_0^y > t$.
    
    Moreover, for the event $\{\tau_{es}^y < \tau_0^y\}$, continuity of the OU Process allows us to write the following decomposition:
    \begin{equation*}
        \{\tau_{es}^y < \tau_0^y\}= \bigcup_{q \in 
        \mathbb{Q}^+}\{X_q^y<-2,\tau_0^y>q\}. 
    \end{equation*}
    Now for each rational $q$, given $\{X_q^y <-2,\tau_0^y>q\}$, by \eqref{pr5} one has $\bar{\tau}_0^y>q$, so that
    \begin{equation*}
        \sqrt{2} B_q^y = X_q^y + \epsilon \int_{0}^{s}X_h^y dh< X_q^y <-2
    \end{equation*}
    which implies that $\bar{\tau}_{-2}^y<q$. Thus
    \begin{equation*}
        \{X_q^y <-2,\tau_0^y>q\} \subset \{\sqrt{2} B_q^y 
        <-2,\bar{\tau}_0^y>q\}, 
    \end{equation*}
    and
    \begin{equation}\label{pr6}
        \{\tau_{es}^y < \tau_0^y\} \subset \{\bar{\tau}_{-2}^y < 
        \bar{\tau}_0^y\}.  
    \end{equation}
    By \eqref{pr5} and \eqref{pr6}, one can bound $A_1$ from below by
    \begin{equation}\label{pr7}
        \mathbb{P}\left(\bar{\tau}_0^y < \bar{\tau}_{-2}^y \wedge 
        1\right)\geq \mathbb{P}\left(\bar{\tau}_0^{-1} < 
        \bar{\tau}_{-2}^{-1} \wedge 1\right) \geq 
        \frac{1}{2}\mathbb{P}\left(|N(0,2)|>1\right)>0.
    \end{equation}
    Now for
    \begin{equation*}  
        A_2 = \min_{t \in [2,3]}\mathbb{P}\left(\tau_{es}^0>t, 
        \int_{0}^{t}\mathbb{I}_{\{X_s^0 \in [0,2]\}}ds \geq 
        \frac{t}{2}\right), 
    \end{equation*}
    by symmetry we have
    \begin{equation}\label{pr8}
        \mathbb{P}\left(\tau_{es}^0>t, \int_{0}^{t}\mathbb{I}_{\{X_s^0 
        \in [0,2]\}}ds \geq \frac{t}{2}\right) \geq 
        \frac{1}{2}\mathbb{P}\left(\tau_{es}^0 > t\right) \geq 
        \frac{1}{2}\mathbb{P}\left(\tau_{es}^0 > 3\right).
    \end{equation}
    Now we note that $X_t^0 = \sqrt{2} e^{-\epsilon t}\int_{0}^{t}e^{\epsilon s} dB_s$. So for any t,
    \begin{equation}\label{pr9}
        \mathbb{P}\left(\tau_{r,+}^0 \wedge \tau_{r,-}^0 >t \right)\geq 
        \mathbb{P}\left(\max_{s \leq t}|\sqrt{2} \int_{0}^{s}e^{\epsilon h}
        dB_h|< 1 \right).
    \end{equation}
    Note that $\int_{0}^{s}e^{\epsilon h} dB_h$ is a martingale. By Doob's Maximum Theorem, there exists a $t_0$,
    % $\exists \ t_0 >0$,
    s.t. the RHS of \eqref{pr9} is greater than $1/2$. By \eqref{pr6}, one has
    \begin{equation}\label{pr10}
        \mathbb{P}\left(\tau_0^1 \leq \tau_2^1 \right)\geq 
        \mathbb{P}\left(\bar{\tau}_0^1 \leq \bar{\tau}_2^1\right)= 
        \frac{1}{2}. 
    \end{equation}
    Thus by \eqref{pr9}, \eqref{pr10} and strong Markov property,
    \begin{equation}\label{pr11}
        \mathbb{P}\left(\tau_{es}^0 >3 \right)\geq 
        (\frac{1}{4})^{[\frac{3}{t_0}]+1}>0.
    \end{equation}
    Note that \eqref{pr7} and \eqref{pr11} do not depend on $y$. Let 
    \begin{equation*}
        P_0 = 
        \frac{1}{2}\mathbb{P}\left(|N(0,2)|>1\right)
        (\frac{1}{4})^{[\frac{3}{t_0}]+1}
         >0.
    \end{equation*}
    The proof of Lemma \ref{lem36} is complete.
\end{proof}

Now we have all the tools in place and now conclude the proof of Proposition \ref{thm31}. 
% and give the proof of Proposition \ref{thm31}.
\begin{proof}
    Firstly, we define three kinds of non-decreasing sequence of stopping 
    times. Let $\Gamma_0^1=\Gamma_0^2=\Gamma_0^3=0$, for any integer $n 
    \geq 1$, define
    \begin{equation*}
        \Gamma_n^1 = \inf \{t \geq \Gamma_{n-1}^3,\quad X_t = \pm 1\},
    \end{equation*}
which means the first time t such that $X_t$ reach $1$ or $-1$ after 
% $t \geq
$\Gamma_{n-1}^3$.
    \begin{equation*}
        \Gamma_n^2 = \Gamma_n^1 +3 \wedge \inf \{t \geq \Gamma_{n}^1,\quad 
        X_t = \pm 2\},
    \end{equation*}
where $\wedge$ means choosing the smaller one.
    \begin{equation*}
        \Gamma_n^3 = \inf \{t \geq \Gamma_{n}^2,\quad X_t = \pm 2\}.
    \end{equation*}
    Then we define a a sequence of Bernoulli r.v. as follows
    \begin{equation*}
        Id(n)=\mathbb{I}_{\{\Gamma_n^2=\Gamma_n^1+3,
        \int_{\Gamma_n^1}^{\Gamma_n^2}\textbf{1}_{\{X_t
         \in [0,2]\}}dt>1\}}.
    \end{equation*}
    Finally, let $N(0)=0$, and for all $k \geq 1$,
    \begin{equation*}
        N(k)=\inf\{n: \sum_{m=1}^{n}Id(m)=k\}.
    \end{equation*}
    By Lemma \ref{lem34}-\ref{lem36} and strong 
    Markov property, all the stopping times above are with probability $1$ 
    finite. Besides, $Id(n)$ ($n=1,2,...$) forms an i.i.d. sequence. Thus $N(1) 
    \sim G(p)$ for some $p \geq P_0$, and
    \begin{equation*}
        \begin{matrix}
            N(1), & N(2)-N(1), & \ldots, & N(k)-N(k-1), & \ldots ,\\
            \Gamma_{N(1)}^3, & \Gamma_{N(2)}^3-\Gamma_{N(1)}^3, & \ldots, & 
            \Gamma_{N(k)}^3-\Gamma_{N(k-1)}^3, & \ldots.
        \end{matrix}
    \end{equation*}
    both form i.i.d. sequences.
    
    For $\Gamma_{N(1)}^2$, given $N(1)=m$, again by strong Markov property,
    \begin{equation*}
        \Gamma_1^3,\ \Gamma_2^3-\Gamma_1^3,\ \ldots,\ 
        \Gamma_{m-1}^3-\Gamma_{m-2}^3,\ \Gamma_m^3-\Gamma_{m-1}^3
    \end{equation*}
    are independent and uniformly exp integrable by Lemma \ref{lem34} and Lemma \ref{lem35} which indicates $\exists\ c_0 < \infty$, s.t.
    \begin{equation*}
        \mathbb{E}\left[\exp(\epsilon_2 
        \left(\Gamma_{k}^3-\Gamma_{k-1}^3\right) \right] \leq c_0,
    \end{equation*}
    where $\epsilon_2=\min\{\frac{\epsilon_1}{2},
    \frac{1}{2}\log(\frac{1}{1-P_0})\}$
    ($\epsilon_1$ is in Lemma \ref{lem34}, $P_0$ is from Lemma \ref{lem35}).
    Thus by the same argument as Corollary \ref{cor31}, $\Gamma_{N(1)}^3$ is exponentially integrable.

    Define $\Gamma^{'}=\frac{1}{C_+}\Gamma$. Recall that $\Gamma \sim \exp(1)$, we have $[\Gamma^{'}]+1$ is a geometric r.v. By definition
    \begin{equation}\label{pr12}
        \Gamma_{N([\Gamma^{'}]+1)}^3 = \sum_{n=1}^{[\Gamma^{'}]+1} 
        \left(\Gamma_{N(n)}^3 - \Gamma_{N(n-1)}^3\right)>T. 
    \end{equation}
    \eqref{pr12} holds since for each $n \geq 1$, 
    $\int_{0}^{t}\mathbf{1}_{\{X_s \in [0,2]\}}ds$ will increase at least 
    $1$ between $\Gamma_{N(n-1)}^3$ and $\Gamma_{N(n)}^3$. Thus the whole 
    proof of Proposition \ref{thm31} is concluded by 
    Lemma \ref{lem32}.
\end{proof}

%%%%%%%%%%%%%%%%%%%%%%%%%%%%%%%%%%%%%%%%%%%%%%%%%
\subsection{Power law decay in the mesoscopic time scale}
%middle}
In Proposition \ref{thm31} we have shown that the 
distribution of $T$ will eventually have an exponential decay as its 
asymptotic. For an OU process $X_t$ in \eqref{1} with $X_0 = -2$ 
where $\epsilon \ll 1$. we will show that $P(T>t)$ is at least of order $O(\frac{1}{t^{1/2}})$ in a appropriate intermediate
%proper 
time scale which is polynomial with respect to $\epsilon$.
% depending on $\epsilon$.
That is
\begin{prop}\label{thm32}
    For any $\alpha \in (0,\frac{1}{6})$ there is a $M=M(\alpha)< \epsilon^{-\frac{1}{2}}$ depending only on $\alpha$ such that
    \begin{equation}\label{pr13}
        \mathbb{P}\left(T>t\right) \geq 
        \frac{1}{2t^{\frac{1}{2}+\alpha}}, \quad \forall t \in 
        [M,{\epsilon}^{- \frac{1}{2}}],
    \end{equation}
for all sufficiently small $\epsilon \ll 1$.
\end{prop}
To prove this theorem, we firstly define $\hat{T}=\inf \{t \geq 0,X_t \geq 
0\}$, which is smaller than $T$ by definition. Thus it suffices to prove 
\eqref{pr13} for $\hat{T}$. Define that 
\begin{equation*}
    \bar{\tau}_x^{-2}=\inf \{t \geq 0, \sqrt{2} B_t^{-2}=x\}.
\end{equation*}
The following lemma is crucial for our proof:
\begin{lem}\label{lem37}
For all $2<x<\epsilon^{-\frac{1}{3}}$ and 
$0<t\leq\epsilon^{-\frac{1}{2}}$ we have 
\begin{equation}\label{lem7}
    \mathbb{P}\left( \bar{\tau}_{-1}^{-2} > \Bar{\tau}_{-x}^{-2}>t,\hat{T}< t 
    \right)=0,
\end{equation}
for all sufficiently small $\epsilon \ll 1$.
\end{lem}
\begin{proof}
Let $\bar{T}=\inf \{t:X_t<-x\}$. Then
\begin{equation*}
    \mathbb{P}\left(\bar{\tau}_{-1}^{-2}>\bar{\tau}_{-x}^{-2}>t,
    \hat{T}<t\right) = I_1 + I_2.
\end{equation*}
We denote
\begin{equation*}
    I_1 = \mathbb{P}\left(\bar{\tau}_{-1}^{-2}>\bar{\tau}_{-x}^{-2}>t,
    \hat{T}<t,\bar{T}>\hat{T}\right),\quad 
    I_2 = \mathbb{P}\left(\bar{\tau}_{-1}^{-2}>\bar{\tau}_{-x}^{-2}>t,
    \hat{T}<t,\bar{T}<\hat{T}\right).
\end{equation*}
For $I_1$, we have
\begin{equation*}
    \left\{\bar{\tau}_{-1}^{-2}>\bar{\tau}_{-x}^{-2}>t,\hat{T}<t,
    \bar{T}>\hat{T}\right\} = \bigcup_{\substack{q \in \mathbb{Q}^+ q < 
    t}}\left\{X_q>0,\bar{T}>q,\bar{\tau}_{-1}^{-2}>\bar{\tau}_{-x}^{-2}>t\right\}.
\end{equation*}
However, since we almost surely (a.s.) have
\begin{equation*}
    X_q = \sqrt{2} B_q^{-2}-\epsilon \int_{0}^{q} X_sds.
\end{equation*}
So in event 
$A_q^1=\{X_q>0,\bar{T}>q,\bar{\tau}_{-1}^{-2}>\bar{\tau}_{-x}^{-2}>t\}$, 
a.s. we have
\begin{align*}
    X_q &={} \sqrt{2} B_q^{-2}-\epsilon \int_{0}^{q} X_sds
         \leq -1 - \epsilon \int_{0}^{q}X_s ds
         \leq{} -1 + \epsilon \int_{0}^{t} -X_s ds \\
        &\leq{} -1+ {\epsilon}^{-\frac{1}{6}}<0,
\end{align*}
which contradicts with its definition $X_q>0$ and implies 		
$\mathbb{P}(A_q^1) \equiv 0$ and thus $I_1 = 0$.
            
As for $I_2$, similarly
\begin{equation*}
    \{\bar{\tau}_{-1}^{-2}>\bar{\tau}_{-x}^{-2}>t,\hat{T}<t,
    \bar{T}<\hat{T}\}= \bigcup_{\substack{q \in \mathbb{Q}^+ q < 
    t}}\left\{X_q<-x,\hat{T}>q,\bar{\tau}_{-1}^{-2}>\bar{\tau}_{-x}^{-2}>t\right\}.
\end{equation*}
In the event $A_q^2 = 
\{X_q<-x,\hat{T}>q,\bar{\tau}_{-1}^{-2}>\bar{\tau}_{-x}^{-2}>t\}$, we know 
that $\bar{\tau}_{-x}^{-2}>t$ and $q<t$. As a result, we have  
$\sqrt{2} B_q^{-2} > -x$. So we a.s. have
\begin{equation*}
    X_q = \sqrt{2} B_q^{-2}-\epsilon \int_{0}^{q} X_sds \geq -x - \epsilon \int_{0}^{q} X_sds > -x,
\end{equation*}
which contradict with its own definition and implies $\mathbb{P}(A_q^2) 
\equiv 0$ and thus $I_2=0$. Together we conclude the proof of Lemma \ref{lem37}.
\end{proof}
\textbf{Proof of Proposition \ref{thm32}}:
    For $2^{\frac{1}{\frac{1}{2}+\alpha}}<t<{\epsilon}^{-\frac{1}{2}}$, 
    we let $x= 
    t^{\frac{1}{2}+\alpha}$ which satisfies 
    $2<x<{\epsilon}^{-\frac{1}{3}}$. So that with 
    the help of 
    Lemma \ref{lem37} we could 
    get a lower bound of $\mathbb{P}(\hat{T}>t)$, which is
    \begin{equation*}
        \mathbb{P}\left(\hat{T}>t\right) \geq 
        \mathbb{P}\left(\bar{\tau}_{-1}^{-2}>\bar{\tau}_{-x}^{-2}>t\right).
    \end{equation*}
    Besides from the fact that $\mathbb{P}(\bar{\tau}_{-1}^{-2}>\bar{\tau}_{-x}^{-2}) = 
    \mathbb{P}(\bar{\tau}_{-1}^{-2}>\bar{\tau}_{-x}^{-2}>t)+\mathbb{P}
    (\bar{\tau}_{-1}^{-2}>\bar{\tau}_{-x}^{-2} \leq t)$,
     we have the following inequality			
    \begin{equation*}	
    \mathbb{P}\left(\bar{\tau}_{-1}^{-2}>\bar{\tau}_{-x}^{-2}>t\right)	\geq 
    \mathbb{P}\left(\bar{\tau}_{-1}^{-2}>\bar{\tau}_{-x}^{-2}\right)-\mathbb{P}\left(\bar{\tau}_{-x}^{-2}
    \leq t\right) = J_1 - J_2.
    \end{equation*}
        Then $J_1=\frac{1}{x-1}$ by Optional Stopping Theorem. And again, 
        by reflection principle
    \begin{equation*}
            \mathbb{P}\left(\bar{\tau}_{-x}^{-2} \leq t\right)= 
            \mathbb{P}\left(\bar{\tau}_{x-2}^{0} \leq t\right) 
            = 2\mathbb{P}\left(\sqrt{2} B_t > t^{\frac{1}{2}+\alpha}-2\right) \leq \exp(-t^{2\alpha}).
    \end{equation*}
        Let $M<\infty$ s.t. $\exp(-t^{2\alpha})<\frac{1}{2}t^{-\frac{1}{2}-\alpha} \quad \forall t \geq M$. We conclude the proof of Theorem \ref{thm32}.

\section{Numerical simulations} \label{sec:5}
Several numerical tests are presented to 
verify the main results, and a few numerical explorations are carried out with more general rate functions. The algorithm is shown in subsection \ref{subsection1} and the case when $\epsilon=0$ is considered in subsection \ref{subsection2}. Then the transitional phenomena with nonzero $\epsilon$ are shown in subsection \ref{subsection3} and \ref{subsection4}, for $\Lambda$ being respectively a piece-wise constant function and smooth functions with a fast transition at $0$.

\subsection{The algorithm}\label{subsection1}
 Each sample is represented by its position $X^n$ and the run time $\tau^n$, where the superscript $n$ is the index of the sample. For each time step, $X^n$ is updated by the Euler-Maruyama method. 
 %The integral of $\Lambda(X^n)$ in \eqref{eq:1passagetime} is updated by Euler method.  
 In all simulations in the subsequent part, $10^5$ samples are tracked and each sample evolves by the algorithm in Section \ref{subsection1} with $\Delta t=10^{-3}$. We denote the numerical approximations of $X^n(k\Delta t)$ and $I^n(k\Delta t)$ by $X_k^n$ and $I_k^n$, respectively. Here we show the details of the algorithm.

\textbf{Initialization.}\quad  The initial values  $X^n$, $\tau^n$ and $I^n$ for all samples are $0$. Set the value of $\epsilon$ in \eqref{1} and the function $\Lambda$. We generate a series of independent random numbers $\Gamma \sim \exp(1)$ and a null matrix $D$ to store the waiting time of all samples. 

\textbf{Time evolution.}\quad For each time step $k$ (=1 initially), we perform the following calculations repeatedly until the termination condition is satisfied:
\begin{itemize}
    \item[1)] Update sample position. Generate a random number denoted by $\Delta B_t$ using the normal distribution $N(0,\Delta t)$. Use the Euler-Maruyama method to update $X^n$
    \begin{equation*}
        X_{k+1}^n = X_k^n - \epsilon X_k^n \Delta t + \sqrt{2} \Delta B_t.
    \end{equation*}
    \item[2)] Update $I^n$. Update the value of $I^n$ by %Euler method.
    \begin{equation*}
        I_{k+1}^n = I_k^n + \Lambda(X_{k+1}^n)\Delta t.
    \end{equation*}
    \item[3)] Update $\tau^n$. Set $\tau^n_{k+1} = \tau^n_{k}+\Delta t$.
    \item[4)] If $I_k^n \geq \Gamma$, store the waiting time for particle $n$ by setting $D=[D,\tau^n_{k+1}]$ and terminate the loop. Otherwise, set $k \gets k+1$, and go back to Step 1).
\end{itemize}

%\begin{algorithm}
%    \caption{} 
%    \label{alg1} 
%    \begin{algorithmic}
%        \REQUIRE For the $n$-th particle
%        \STATE initial position $X^n$; sample location $X^n$; the integral $I^n$; total time $T_t$; time step $\Delta t$; 
%        \STATE exponential distribution rate $\lambda_{\Gamma}$; Drift parameter $\epsilon$; data null vector $D$; Duration time $\tau$;
%        \STATE Generate one exponential distribution random number $\Gamma$ with the rate $\lambda_{\Gamma}$
 %       \REPEAT
 %       \STATE Generate one normally distributed random number $\Delta B_t$ from $N(0,\Delta t)$
 %       \STATE $X^n \gets X^n + \sqrt{2} \Delta B_t + \epsilon X^n \Delta t$
 %       \STATE $I^n \gets I^n + \Lambda(X^n)\Delta t$
 %       \STATE $\tau^n \gets \tau^n + \Delta t$
 %       \UNTIL $I^n \geq \Gamma$
 %       \STATE $D \gets [D,\tau^n]$
 %   \end{algorithmic}
%\end{algorithm}

\subsection{The case when $\epsilon=0$.}\label{subsection2}
In this case, $X_t = B_t$ is a Brownian motion. The rate function $\Lambda$ is chosen to be the piece-wise constant function as in \eqref{Lambdaxspecial}. We evolve each sample by the algorithm in Section \ref{subsection1}.

After getting the $10^5$ waiting times $\{\tau^n\}_{n=1}^{10^5}$ stored in $D$, letting $T_t=\max_{n=1,\cdots,10^{5}}\{\tau^n\}$, we divide $[0,T_t]$ into $N_t$ intervals of the same size $100$ and count the number of $\tau^n$ that fall in each interval. The number of $\tau^n$ in the $i$-th interval is denoted by $N_i$. For $i=1,\cdots,N_t$, we plot $(\log[100(i-1)+1],\log[N_i/10^5])$ in Fig. \ref{fig_BM_powerlaw} which is the PDF of the waiting time distribution in a log-log scale. The black dashed line in Fig.\ref{fig_BM_powerlaw} is a straight line whose slope is $-1.5$, it fits well with the PDF of $\tau$ in the interval $[4,7]$ which is consistent with the theoretical prediction that $n(t)$ decays as $t^{-1.5}$ when $t$ is large enough.

\begin{figure}[!htbp]  
      \centering  
      \includegraphics[width=0.5\textwidth]{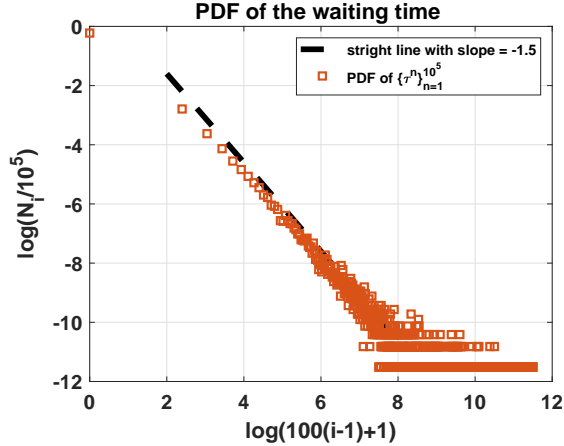} 
      \caption{PDF of the waiting time distribution (squares). We take $\Lambda$ in \eqref{Lambdaxspecial} and run $10^5$ samples. The slope of the straight line is $-1.5$.}
      \label{fig_BM_powerlaw}
\end{figure}

%  \begin{algorithm}
%     \caption{} 
%     \label{algBM} 
%     \begin{algorithmic}
%         \REQUIRE For the $n$-th particle
%         \STATE initial position $X^n$; sample location $X^n$; the integral $I^n$; total time $T_t$; time step $\Delta t$; 
%         \STATE Drift parameter $\epsilon$; data null vector $D$; Duration time $\tau$;
%         \REPEAT
%         \STATE Generate one random number $\Delta B_t$ according to the normal distribution $N(0,\Delta t)$
%         \STATE $X^n \gets X^n + \Delta B_t$
%         \IF{$X^n \geq 0$}
%         \STATE $\tau \gets \tau + \Delta t$
%         \ENDIF
%         \UNTIL $\tau \geq 1$
%         \STATE $D \gets [D,\tau]$
%     \end{algorithmic}
% \end{algorithm}

\subsection{The case when $\epsilon\neq 0$ and $\Lambda$ being piece-wise constant}\label{subsection3}
Let the rate function $\Lambda$ be a piece-wise function as in \eqref{Lambdaxspecial}, then $I^n$ is the duration time when $X^n$ is greater than $0$. We take different values of $\epsilon$ such that $\epsilon=4.0000 \times 10^{-3}, 1.0000\times10^{-3},\, 2.5000\times10^{-4},\, 6.2500\times10^{-5},\, 1.5625 \times 10^{-5}$ and get $10^5$ waiting times $\{\tau^n\}_{n=1}^{10^5}$ for every $\epsilon$. 

Similar to the test in the previous subsection, we obtain $\{\tau^n\}_{n=1}^{10^5}$ and $\{N_i\}_{i=1}^{N_t}$ for each $\epsilon$. Then $S_{Ni} = \sum_{j=i}^{N_t} N_j$ is the number of $\tau^n$'s that fall in the interval $[100(i-1),T_t]$($i=1,...,N_t$). In Fig.\ref{figOU5driftspowerlaw}, we plot $(\log[100(i-1) + 1],\log[S_{Ni}/10^5])$ which is the cumulative density function (CDF) of the waiting time distribution in a log-log scale. We can observe from Fig.\ref{figOU5driftspowerlaw} that all waiting time distributions exhibit transitions from a power law decay to a fast damping tail. %As $\epsilon$ decreases, the interval that exhibits power law decay is longer. 
We observe that as $\epsilon$ decreases, the regimes that exhibit power-law decay in the waiting time extend in size accordingly. 
%  \begin{algorithm} 
%     \caption{} 
%     \label{algpiecewise} 
%     \begin{algorithmic}
%         \REQUIRE  initial location $X_0$; sample location $X$; total time 
%         $T$; time interval $n$; exponential distribution rate $\lambda_{\Gamma}$;
%         normal distribution $N(0,1)$;
%         time step $\Delta t$; sample number $N$; drift parameter $\epsilon$;
%         data null vector $D$; duration $\tau$;
%         \REPEAT
%         \STATE Generate one exponential distribution random number $\Gamma$ with the rate $\lambda_{\Gamma}$
%         \STATE $X \gets X + \sqrt{\Delta t} N(0,2) + c X \Delta t$
%         \IF{$X \geq 0$}
%         \STATE $\tau \gets \tau + \Delta t$
%         \ENDIF
%         \UNTIL $I \geq 1$
%         \STATE $D \gets [D,\tau]$
%     \end{algorithmic} 
% \end{algorithm} 

\begin{figure}[ht]  
          \centering  
          \includegraphics[width=0.5\linewidth]{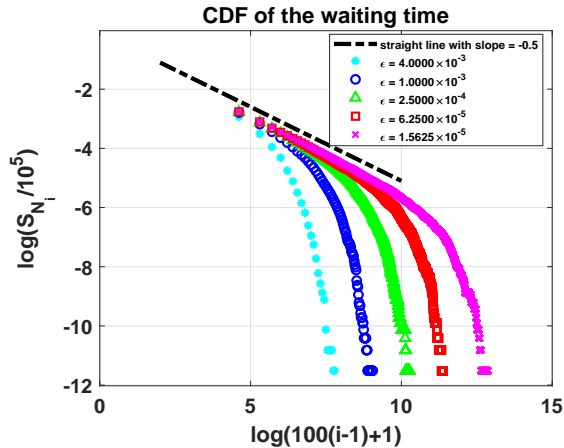} 
          \caption{CDF of the waiting time for different $\epsilon$. We run $10^5$ samples for every $\epsilon$. $\epsilon = 4.0000\times10^{-3}$ (asterisks), $\epsilon = 1.0000\times10^{-3}$ (circles), $\epsilon = 2.5000\times10^{-4}$ (upward triangles), $\epsilon = 6.2500\times10^{-5}$ (squares), $\epsilon = 1.5625\times10^{-5}$ (crosses). The slope of the reference straight line is $-0.5$.}
          \label{figOU5driftspowerlaw}
\end{figure}
\begin{figure}[htbp]
\centering
    \subfloat[]{\includegraphics
    [width = 0.5\textwidth]{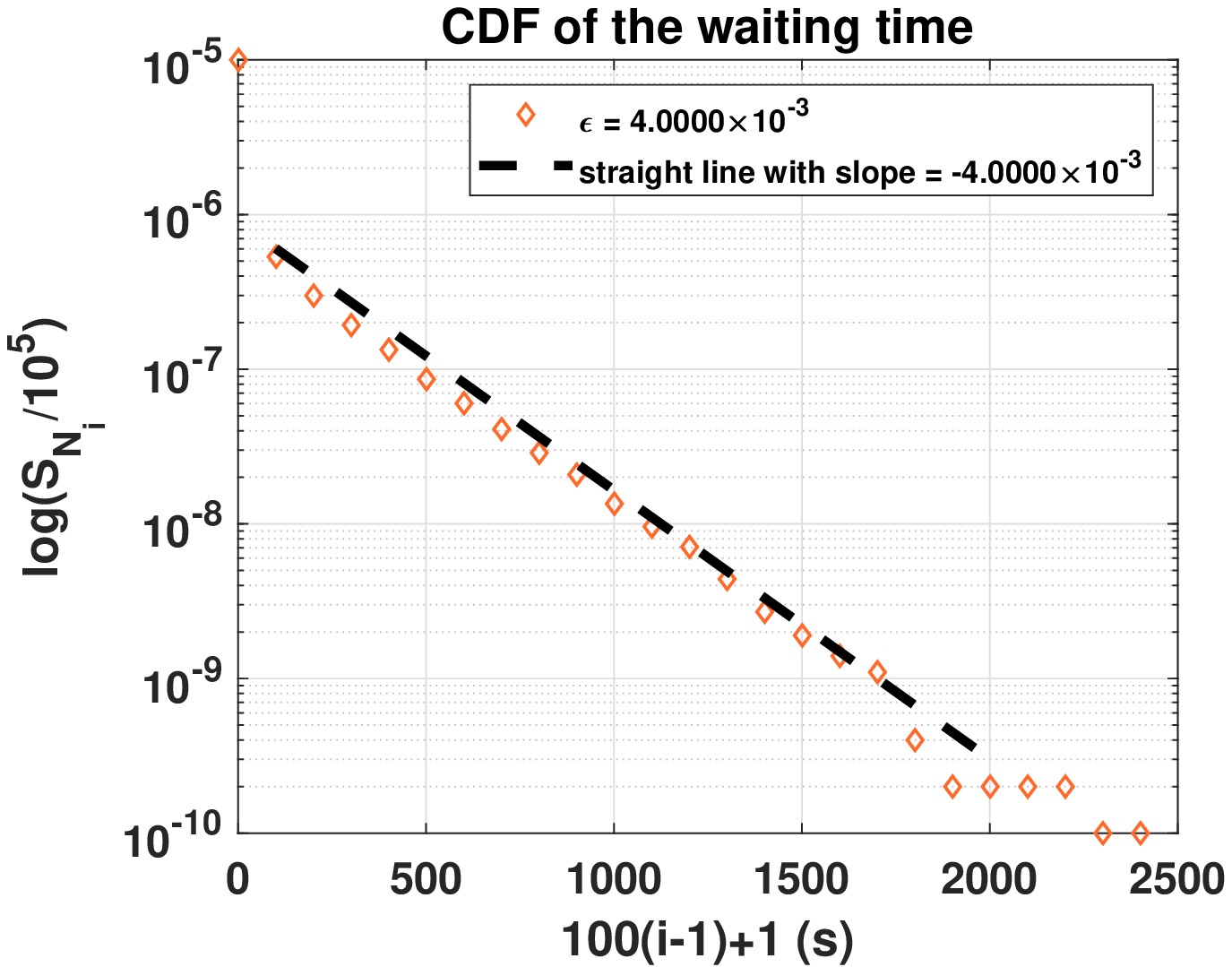}}
    \hfill
    \subfloat[]{\includegraphics
    [width = 0.5\textwidth]{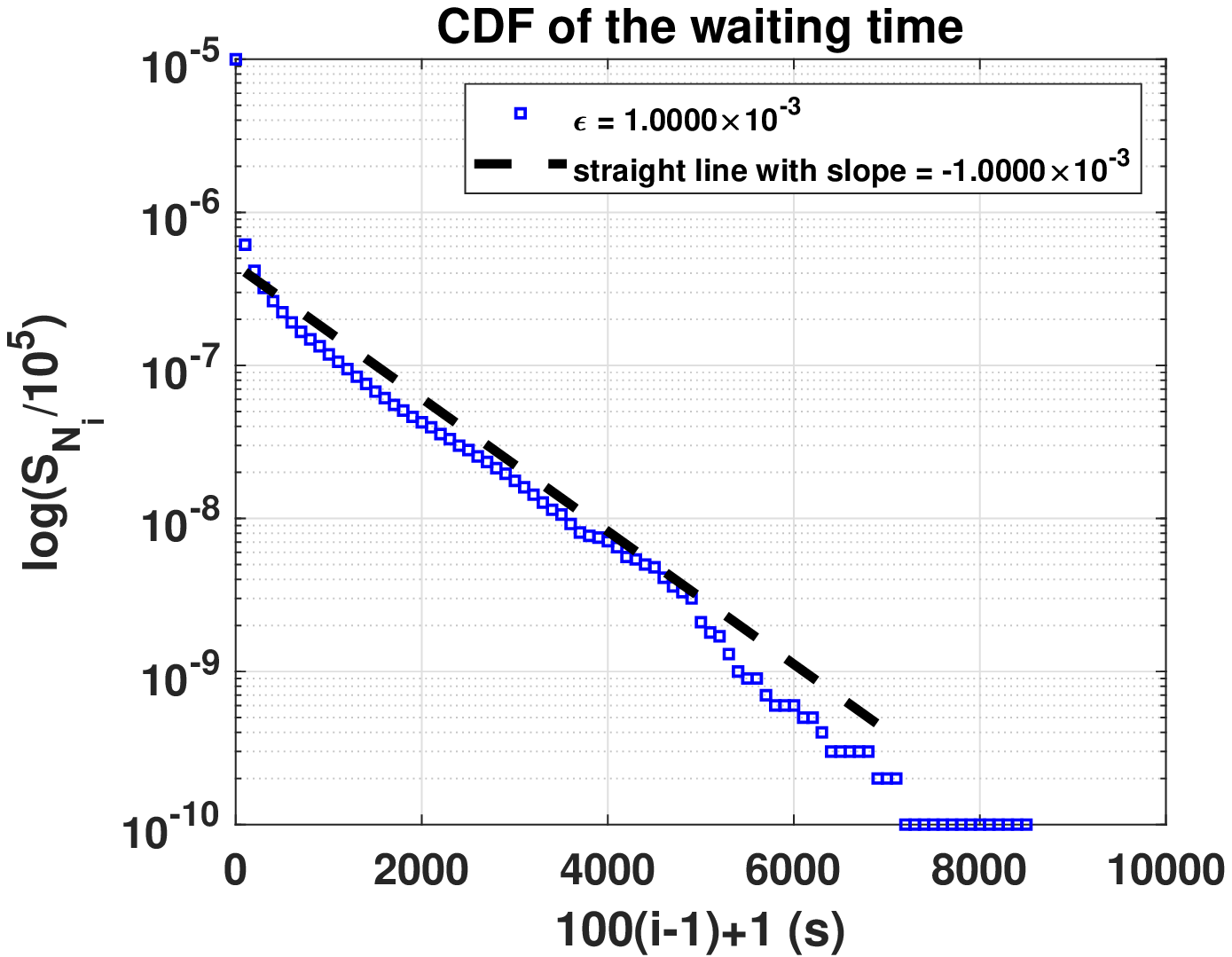}}
    \hfill
    \\
    \subfloat[]{\includegraphics
    [width = 0.5\textwidth]{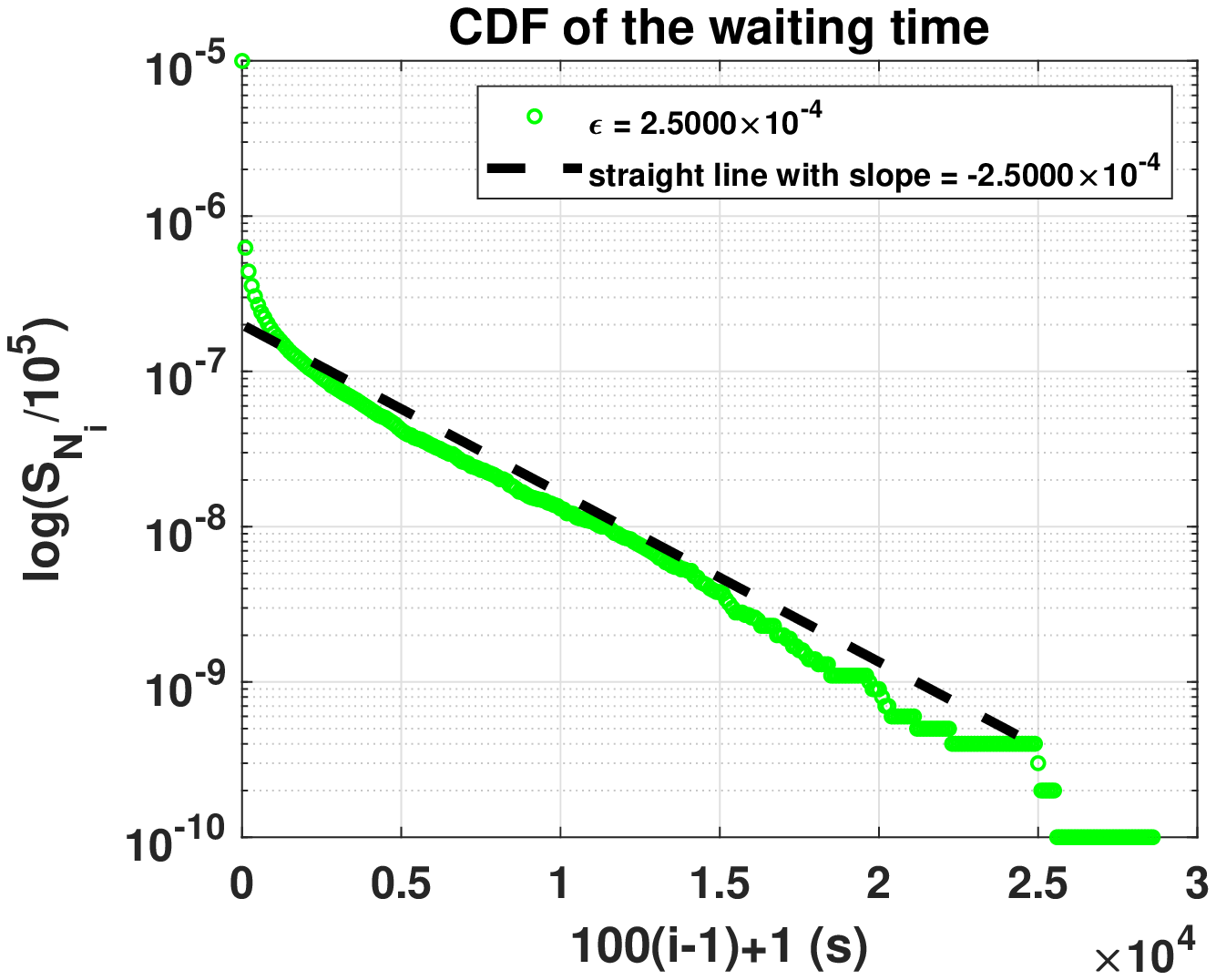}}
    \hfill
    \subfloat[]{\includegraphics
    [width = 0.5\textwidth]{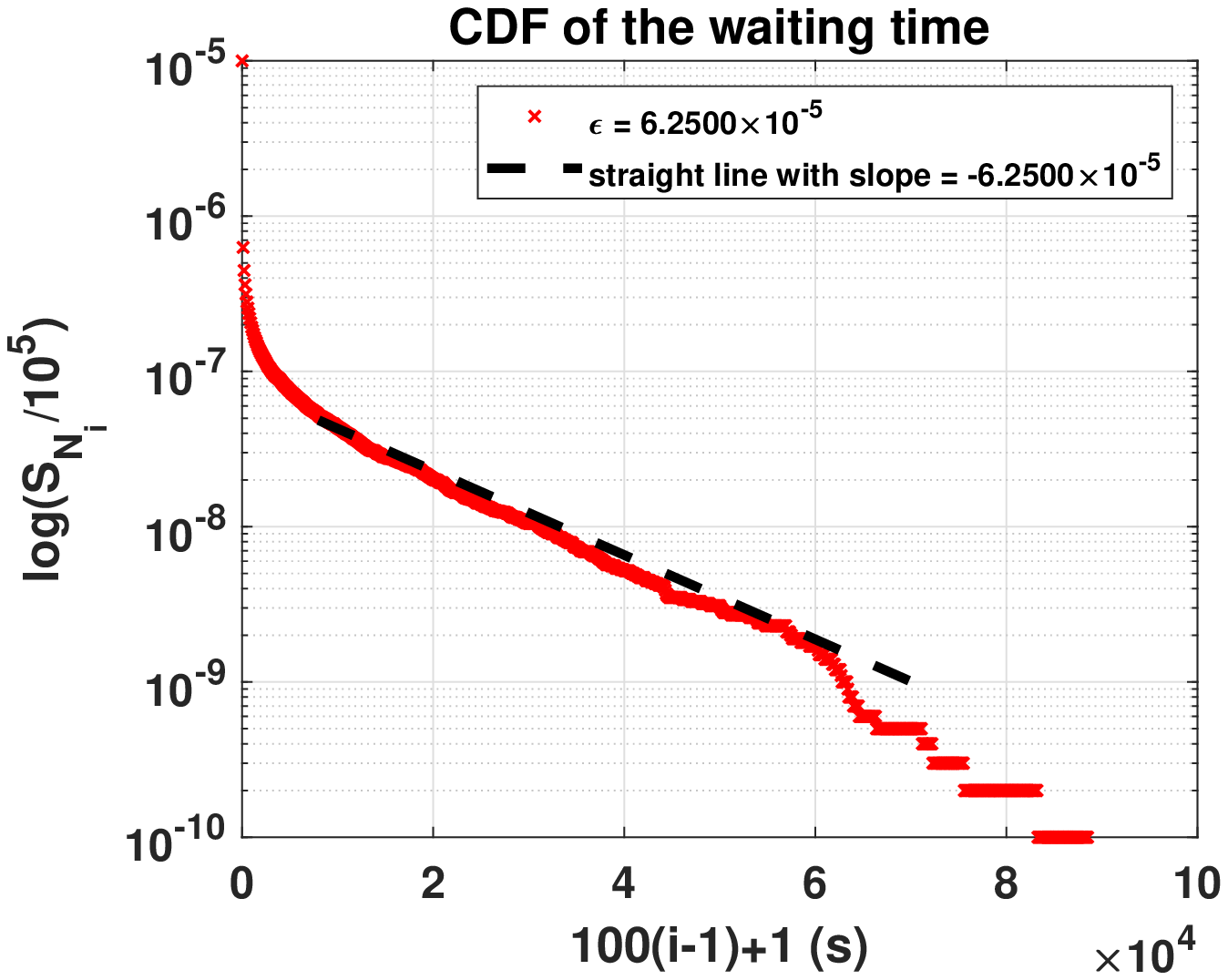}}
    \hfill
    \\
    \subfloat[]{\includegraphics
    [width = 0.5\textwidth]{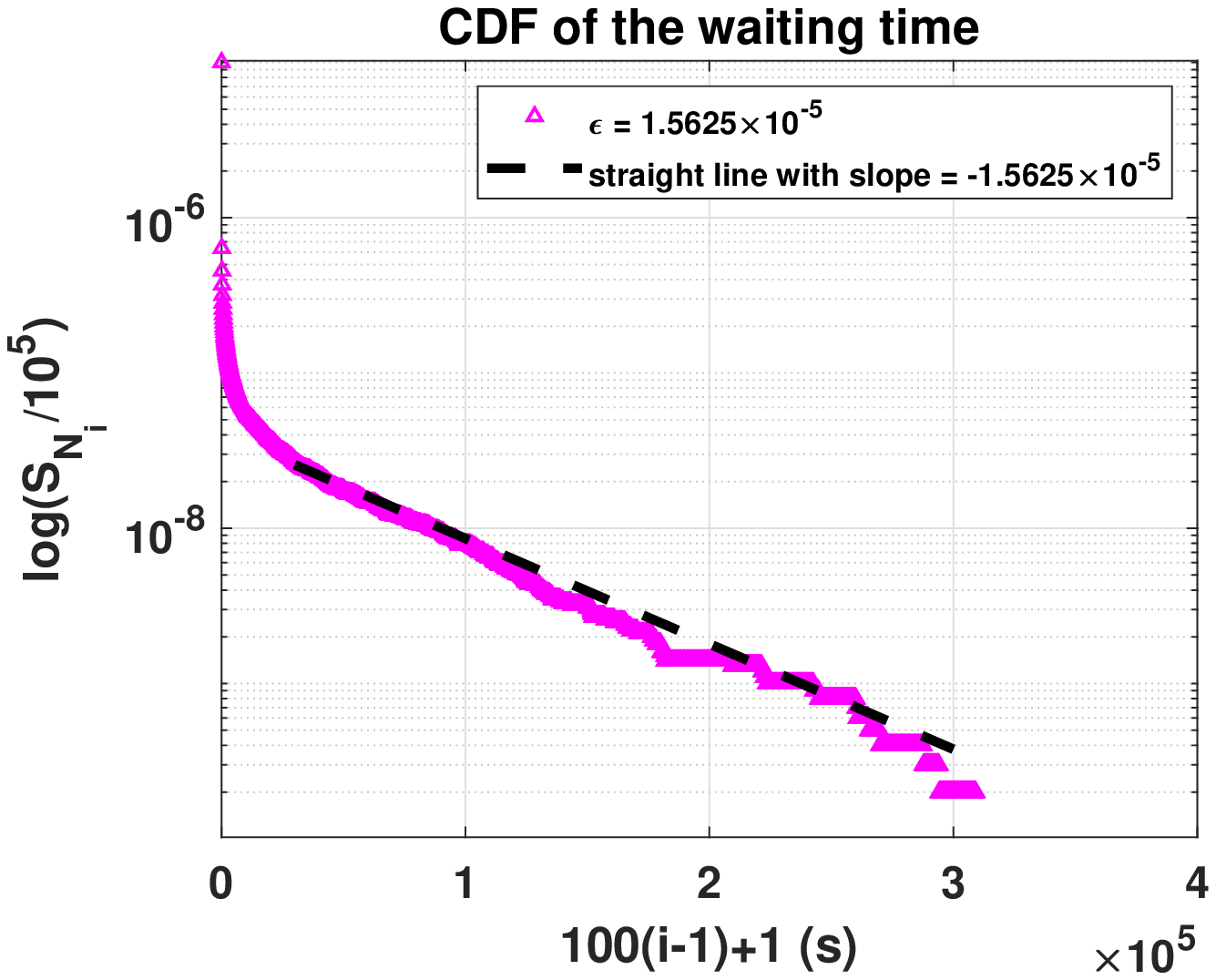}} 
    \caption{CDF of the waiting time for different $\epsilon$. (a) $\epsilon = 4.0000\times10^{-3}$ (orange diamonds), (b) $\epsilon = 1.0000\times10^{-3}$ (blue squares), (c) $\epsilon = 2.5000\times10^{-4}$ (green circles), (d) $\epsilon = 6.2500\times10^{-5}$ (red crosses), (e) $\epsilon = 1.5625\times10^{-5}$ (pink upward triangles).}
    \label{figOU5driftsexp}
\end{figure}
%\textcolor{red}{ Moreover, the colored chain-dotted line decreases rapidly at the tail part and there exists a transition from the power-law decay to a faster decrease. To show the faster decrease more visually, we plot Fig.\ref{figOU5driftsexp}.}

To examine the tail distribution in more detail, we visualize the waiting time distribution in a linear-log scale as shown in Fig.\ref{figOU5driftsexp}. We can observe that for all $\epsilon$, the tail parts can be fitted by straight lines, which indicates that the CDFs decay exponentially fast at the tail parts. The slopes and transition points are listed in Table \ref{table1}. As can be seen in the Table \ref{table1}, the absolute values of the slopes are the same as $\epsilon$. Besides, the transition points are in the same order as $\epsilon^{-1}$ for all five $\epsilon$, which are larger than $\epsilon^{-\frac{1}{2}}$. This numerical observation suggests that the time interval of the power law decay given in part (a) of Theorem \ref{mainthm1} is not optimal.
 \begin{table}[!htbp]
  \begin{center}
    \caption{The slopes and transition points for different $\epsilon$.}
    \label{table1}
    \begin{tabular}{|c|c|c|c|c|} % <-- Alignments: 1st column left, 2nd middle and 3rd right, with vertical lines in between
    \hline
      \textbf{$\epsilon$} & \textbf{slope} & \textbf{$\epsilon^{-\frac{1}{2}}$} & \textbf{$\epsilon^{-1}$} & \textbf{transition point}\\
      \hline
      $1.5625\times10^{-5}$ & $-1.5625\times10^{-5}$ & $2.5298\times10^2$ & $6.4000\times10^4$ & $3.3000\times10^4$\\
      \hline
      $6.2500\times10^{-5}$ & $-6.2500\times10^{-5}$ & $1.2649\times10^2$ & $1.6000\times10^4$ & $9.6000\times10^3$\\
      \hline
      $2.5000\times10^{-4}$ & $-2.5000\times10^{-4}$ & $6.3246\times10^1$ & $4.0000\times10^3$ & $2.2000\times10^3$\\
      \hline
      $1.0000\times10^{-3}$ & $-1.0000\times10^{-3}$ & $3.1622\times10^1$ & $1.0000\times10^3$ & $7.0000\times10^2$\\
      \hline
      $4.0000\times10^{-3}$ & $-4.0000\times10^{-3}$ & $1.5811\times10^1$ & $2.5000\times10^2$ & $1.0000\times10^2$\\
      \hline
    \end{tabular}
  \end{center}
\end{table}
% As in Fig.\ref{figOU5driftsexp}, for different $\epsilon$, the slopes of the straight lines are respectively $-4\times10^{-3}$, $-10^{-3}$, $-2.5\times10^{-4}$, $-7.5\times10^{-5}$ and $-3.125\times10^{-5}$ in the interval $[200s,2000s]$. In Fig.\ref{figOU5driftsexp}(b), the blue line fits well with the black dotted line whose slope is $-10^{-3}$ in [1000s,5000s]. As seen in Fig.\ref{figOU5driftsexp}(c), in $[1000s,25000s]$ the green line fits well with the straight dotted line with slope $-2.5\times10^{-4}$. The red line goes down roughly like the straight dotted line with slope $-7.5\times10^{-5}$ in $[5000s,45000s]$ in Fig.\ref{figOU5driftsexp}(d). In Fig.\ref{figOU5driftsexp}(e), the pink line decays roughly like the straight dotted line with slope $-3.125\times10^{-5}$ in $[10000s,70000s]$. Fig.\ref{figOU5driftsexp} indicates tail probability of the waiting time has an exponential decay in the tail

\subsection{Different $\epsilon$ for $\Lambda$ being an inverse tangent function}\label{subsection4}

In this subsection, we test with three rate functions $\Lambda(x)$ which are inverse tangent functions with different stiffness. The three rate functions  are
$$\Lambda_1(x) = \arctan(10^5x) + \frac{\pi}{2},\quad \Lambda_2(x) = \arctan(10^3x) + \frac{\pi}{2},\quad \Lambda_3(x) = \arctan(10x) + \frac{\pi}{2},$$ which are plotted in Fig.\ref{fig_3atanfcns}. Clearly, $\Lambda_1$ has the sharpest transition from $0$ to $\pi$, while $\Lambda_2$ increases less sharply and $\Lambda_3$ has the mildest transition. Besides, there is a positive lower bound $\frac{\pi}{2}$ when $x>0$ for all these three cases. However, we notice that these rate functions do not satisfy \eqref{Lambdax} since they do not strictly take $0$ value when $x<0$, and their properties are beyond our theoretical analysis. But these rate functions 'morally' satisfy the modeling principle as they are almost $0$ for $x<0$ up to a removal of a small neighborhood around $x=0$. Thus, it is worth testing if these models also exhibit similar transitional behavior in terms of the waiting time distribution.

\begin{figure}[ht]  
          \centering  
          \includegraphics[width=0.5\linewidth]{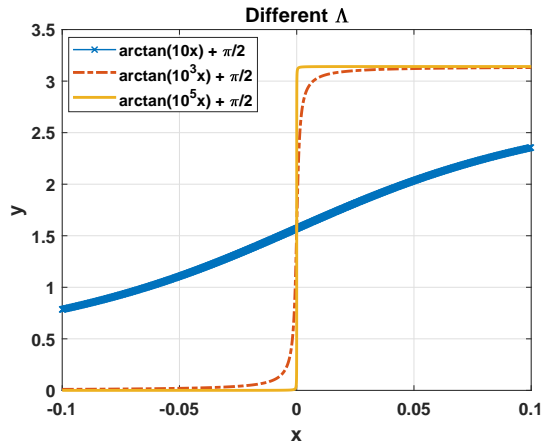} 
          \caption{The curves of $\Lambda_i(x)$ ($i=1, 2, 3$). $\Lambda_1(x)$ (yellow solid line), $\Lambda_2(x)$ (red dotted line), $\Lambda_3(x)$ (blue line).}
          \label{fig_3atanfcns}
\end{figure}

\begin{figure}[ht]  
    \centering  
    \includegraphics[width=0.5\linewidth]{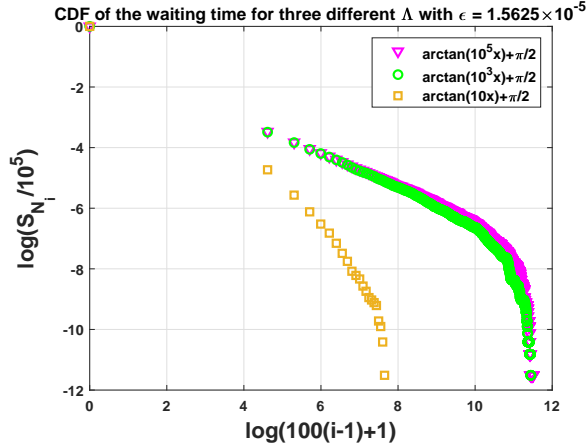} 
    \caption{CDF of the waiting time for $\Lambda_i$ ($i=1, 2, 3$). $\Lambda_1$ (pink downward triangles), $\Lambda_2$ (green circles), $\Lambda_3$ (yellow squares). We take $\epsilon = 1.5625\times10^{-5}$ and run $10^5$ samples for each $\Lambda_i$ ($i=1, 2, 3$). The green and the pink symbols have a transition from power-law decay to a faster decay. The transition for the yellow symbols is not obvious.}
    \label{fig_3atanfcns_powerlaw}
\end{figure}

We take $\epsilon = 1.5625 \times 10^{-5}$ and run $10^5$ samples for the three different $\Lambda(x)$. Similar as in subsection \ref{subsection3}, $10^5$ waiting times $\{\tau^n\}_{n=1}^{10^5}$ as well as $\{N_i\}_{i=1}^{N_t}$ and $\{S_{N_i}\}_{i=1}^{N_t}$ are obtained for each $\Lambda(x)$. We plot ($\log(100(i-1)+1)$, $\log[S_{N_i}/10^5]$) in Fig. \ref{fig_3atanfcns_powerlaw} which is the CDF of the waiting time distribution in a log-log scale. As shown in Fig.\ref{fig_3atanfcns_powerlaw}, the transition from a power-law decay to a more rapid decay is obvious when using $\Lambda_1(x)$ or $\Lambda_2(x)$, but is less clear when using $\Lambda_3(x)$. These results seem to suggest that \eqref{Lambdax} is not a necessary condition for the transitional behavior and a more comprehensive characterization calls for future studies. 

%\textcolor{red}{In Fig.\ref{fig_atanx_5exp}, we plot $(100(i-1) + 1,\log[S_{Ni}/10^5])$ to show the tail distribution of $\Lambda_3(x) = \arctan(x)+\frac{\pi}{2}$ with five different $\epsilon$. In Fig.\ref{fig_atanx_5exp}(a), the solid line fits well with the dashed line $t^{-2.3}$ in the region $[0,60]$. As seen in Fig.\ref{fig_atanx_5exp}(b), the solid line fits well with the dashed line $t^{-2.35}$ in $[0,65]$. In the interval $[0,40]$, the solid line fits well with the dashed line $t^{-2.4}$ in Fig.\ref{fig_atanx_5exp}(c). Fig.\ref{fig_atanx_5exp}(d) displays that the solid line and the dashed line $t^{-2.4}$ are almost coincident in the interval $[0,55]$. As illustrated in Fig.\ref{fig_atanx_5exp}(e), in the interval $[0,45]$, the solid line and the dotted line $t^{-2.45}$ nearly overlap. This phenomenon indicates that the CDF of the waiting time has a power-law decay for $\Lambda(x) = \arctan(x)+\frac{\pi}{2}$.Moreover, in the tail of the CDF of the waiting time, there is no obvious exponential decline.}

\section{Conclusion and discussions}

In this work, we have explored the transitional behavior of the waiting time distribution of a state-dependent jump model, which may be interpreted as an underlying principle for the transitory anomalous diffusion at the macroscopic level.   

The two approaches we have presented naturally complement each other. The asymptotic analysis leads to a fully specified description of the leading order behavior but its application is limited to a few cases. The probability method works for more generic parameter settings, but it only provides estimates in terms of the cumulative density functions, and there is no obvious way to justify the optimality of the derived bounds just by themselves.  It is worth pointing out that, for simplicity of analysis, we have focused on the rate model \eqref{Lambdax} only to demonstrate the key components of the internal state causing the transitional phenomenon, but a series of rate models beyond Assumption \eqref{Lambdax}  have been tested in our numerical experiments as well. And it is interesting to investigate in a more quantitative way how the transition rate affects the waiting time distribution and whether possible phase transitions exist.

This work also establishes a theoretical platform and a practical modeling methodology for incorporating transitional behavior in more complicated physical or biological systems. For example, we may study moving agents which have switching models with transitional waiting times in the future.  

\section*{Acknowledgment}
Z. X. and M. T are partially supported by NSFC 11871340, NSFC12031013, Shanghai pilot innovation project 21JC1403500. Y.Z. was supported by NSFC Tianyuan Fund for Mathematics grant, Project Number 12026606. Z.Z. has received support from  the National Key R\&D Program of China, Project Number 2021YFA1001200, 2020YFA0712000,  and the NSFC grant, Project Number 12031013, 12171013. 

\appendix
\section{The introduction of the Kummer's function}
    Kummer's function of the first kind $M(a,b,z)$ is a solution of the confluent hyper-geometric differential equation 
    \begin{equation*}
        z \frac{d^2w}{dz^2} + (b-z) \frac{dw}{dz} -a w = 0,
    \end{equation*}
    where $a$, $b$ are two constant.
    
    $M(a,b,z)$ has a hyper-geometric series given by
    \begin{equation}\label{kummerfunction_M}
        M(a,b,z) = 1 + \frac{a}{b} z + \frac{a(a+1)}{b(b+1)2!} z^2 + ...
    \end{equation}
    We have 
    \begin{equation}\label{M_0}
        M(a,b,z) \to 1, \quad z \to 0.
    \end{equation}
    When $z \to \infty$, one gets
    \begin{equation}\label{M_infty}
        M(a,b,z) \sim e^z z^{a-b} \frac{\Gamma(b)}{\Gamma(a)}, \quad |\mathbf{ph} z| < \frac{1}{2} \pi,
    \end{equation}
    where $\mathbf{ph} z$ is the principle value of $a$, $-\pi<\mathbf{ph} z<\pi$ and $a \neq 0,-1,-2...$.

\section{The introduction of the parabolic cylinder functions}
Parabolic cylinder functions are solutions of the differential equation
\begin{equation*}
    \frac{d^2y}{dx^2} + (a_0x^2+b_0x+c_0)x = 0,
\end{equation*}
which can be converted into the following form
\begin{equation}\label{PCFstandard}
    \frac{d^2y}{dx^2}-(\frac{1}{4}x^2+a)y=0.
\end{equation}
\eqref{PCFstandard} has two standard solutions \cite{abramowitz1972abramowitz, NISTDLMF} 
\begin{align}\label{standard_U(a,x)}
    U(a,x) &={}  \frac{\sqrt{\pi}}{2^{\frac{a}{2}+\frac{1}{4}}\Gamma(\frac{3}{4}+\frac{a}{2})} e^{-\frac{\epsilon}{4}x^2} M\left(\frac{a}{2}+\frac{1}{4},\frac{1}{2},\frac{1}{2}x^2\right) \notag \\
    &-{} \frac{\sqrt{\pi}}{2^{\frac{a}{2}-\frac{1}{4}}\Gamma(\frac{1}{4}+\frac{a}{2})} x e^{-\frac{\epsilon}{4}x^2} M\left(\frac{a}{2}+\frac{3}{4},\frac{3}{2},\frac{1}{2}x^2\right).
\end{align}
\begin{equation}\label{standard_V(a,x)}
    V(a,x) = \frac{\sin\left[\pi(\frac{1}{4}+\frac{a}{2})\right]Y_1 + \cos\left[\pi(\frac{1}{4}+\frac{a}{2})\right]Y_2}{\Gamma(\frac{1}{2}-a)},
\end{equation}
where 
\begin{align*}
    Y_1 &={} \frac{1}{\sqrt{\pi}}\frac{\Gamma(\frac{1}{4}-\frac{a}{2})}{2^{\frac{a}{2}+\frac{1}{4}}} e^{-\frac{\epsilon}{4}x^2} M\left(\frac{a}{2}+\frac{1}{4},\frac{1}{2},\frac{1}{2}x^2\right), \\
    Y_2 &={} \frac{1}{\sqrt{\pi}}\frac{\Gamma(\frac{3}{4}-\frac{a}{2})}{2^{\frac{a}{2}-\frac{1}{4}}} x e^{-\frac{\epsilon}{4}x^2} M\left(\frac{a}{2}+\frac{3}{4},\frac{3}{2},\frac{1}{2}x^2\right).
\end{align*}
Here, $M$ is the Kummer's function in \eqref{kummerfunction_M}.

We have 
\begin{equation}\label{U_infty}
    U(a,x) \sim e^{-\frac{1}{4}x^2} x^{-a-\frac{1}{2}}\left[1+O(\frac{1}{x^2})\right],\quad x\to\infty,
\end{equation}
\begin{equation}\label{V_infty}
    V(a,x) \sim \sqrt{\frac{2}{\pi}}e^{\frac{1}{4}x^2} x^{a-\frac{1}{2}}\left[1+O(\frac{1}{x^2})\right],\quad x\to\infty.
\end{equation}
and
\begin{equation}\label{U_0}
    \lim_{x \to 0+} U(a,x) = \frac{\sqrt{\pi}}{2^{\frac{a}{2}+\frac{1}{4}}\Gamma(\frac{3}{4}+\frac{a}{2})}.
\end{equation}
For the derivative of $U(a,x)$, one has\cite{NISTDLMF}
\begin{equation}\label{derivate_U}
    U^{'}(a,x) = \frac{1}{2}xU(a,x)-U(a-1,x),
\end{equation}
and
\begin{equation}\label{derivate_U_0}
    \lim_{x \to 0+} U^{'}(a,0+) = -\frac{\sqrt{\pi}}{2^{\frac{a}{2}-\frac{1}{4}}\Gamma(\frac{a}{2}+\frac{1}{4})}.
\end{equation}

\bibliographystyle{siamplain}
\bibliography{ref}

\end{document}